\documentclass{amsart}

\usepackage{amssymb,amsmath,xcolor,graphicx,xspace,colortbl,rotating} 
\usepackage{amsfonts}  
\usepackage{amsmath}  
\usepackage{amscd}
\graphicspath{{Cohomology15_copia_graphics/}{Cohomology15_copia_tcache/}{Cohomology15_copia_gcache/}}
\DeclareGraphicsExtensions{.pdf,.eps,.ps,.png,.jpg,.jpeg}
\theoremstyle{plain}
\newtheorem{theorem}{Theorem}
\newtheorem{lemma}{Lemma}
\newtheorem{corollary}{Corollary}
\newtheorem{proposition}{Proposition}

\theoremstyle{definition}

\newtheorem{problem}{Problem}
\newtheorem{example}{Example}

\theoremstyle{remark}
\newtheorem{remark}{Remark}

\numberwithin{equation}{section}
\def\OO{\mathcal{O}}

\def\II{\mathcal{I}}

\def\N{\mathcal{N}}

\def\PP{\mathbb{P}}
\def\Z{\mathbb{Z}}

\def\T{\mathcal{T}}

\def\op{\operatorname}
\def\rk{\operatorname{rk}}

\begin{document}

\title[Cohomology of Normal Sheaves]{Cohomology of normal bundles of special rational varieties}

\author{Alberto Alzati }
\address{Dipartimento di Matematica Univ. di Milano 
via C. Saldini 50 - 20133 Milano (Italy)}

\email[]{alberto.alzati@unimi.it}

\author{Riccardo Re
}

\address{Dipartimento di Scienza ed alta Tecnologia Univ. dell'Insubria
via Valleggio 11 22100-Como (Italy). }

\email[]{riccardo.re@uninsubria.it}

\urladdr{
}
\date{January, 2nd 2020
}
\dedicatory{
}

\thanks{This work has been done within the framework of the national project
"Geometry on Algebraic Varieties", Prin (Cofin) 2015 of MIUR}

\begin{abstract} We give a new method for calculating the cohomology of the normal bundles over rational varieties which are smooth projections of Veronese embeddings. The method can be used also when the projections are not smooth, in this case it provides informations about the critical locus of maps between projective spaces.
\end{abstract}
\maketitle

\section{Introduction}
The problem of calculating the restricted tangent and especially the normal bundle to a smooth or at most nodal rational curve in a projective space has been very popular since the '80's, when a series of papers was devoted to the problem of classifying the rational curves with normal bundle having a given {\em splitting type}, when pulled back to $\mathbb{P}^1$ by means of the parametrization map, see for example \cite{EiVdV1}, \cite{EiVdV2}, \cite{Sac}, \cite {Verd}. This problem has received renewed attention in recent times in a line of research that answered some of the questions left open in those earlier papers, notably the question of the irreducibility of the Hilbert scheme of smooth rational curves having a given splitting type of the normal bundle, see \cite{A-R1} and \cite{CoRi}.
 In order to classify curves having normal bundle with a given splitting type, one has has to develop general methods of calculating this splitting type from a given embedding $\PP^1\to\PP^s$. This is equivalent to calculating the cohomology groups of any twist of the normal bundle pulled back to $\PP^1$.

In \cite{A-R1} and \cite{A-R0} we introduced a new method to calculate the cohomology groups of the tangent and the normal bundle of smooth rational curves embedded in projective spaces. This method has been useful in some applications developed in those papers and other applications are given in \cite{A-R2} and \cite{A-R-T}. In \cite{CoRi} a very simple method is found for calculating the splitting type of the normal bundle of a curve parametrized by a {\em monomial} map $\PP^1\to\PP^s$. It is natural to try to extend the method of calculation of cohomology from our previous papers to any smooth rational variety of dimension $n \geq 1$.

Let us consider maps $f:\PP^n\to \PP^s$, given by $s+1$ homogeneous, degree $d$ polynomials. In this case we can define the normal bundle $\mathcal{N}_{f}$ of the map $f$ as the cokernel of $\T_{\PP^n}\stackrel{df}\longrightarrow f^\ast \T_{\PP^s}$.  In this article we develop a method to calculate the cohomology of $f^\ast \T_{\PP^s}(-k)$ and $\N_f(-k)$ for any integer $k$. The proofs require some algebraic machinery based on representation theory and developed in section \ref{sec:background}.

When $X:=f(\PP^n)$ is smooth and $f$ is an embedding, $X$ is the smooth projection of some $d$-Veronese embedding of $\mathbb{P}^{n}$ in $\mathbb{P}^{N}$, where $N =\binom{d +n}{d}-1$. In this case $f^\ast \T_{\PP^s}$ is the restriction of the tangent bundle of $\PP^s$ to $X$ and $\N_f$ can be identified with the normal bundle of $X$ in $\PP^s$. So we get a method for calculating the cohomology of the normal and the restricted tangent bundle of smooth projections of $d$-Veronese embeddings, see Theorem \ref{teotg} in section \ref{Section restricted tg}  and Theorem \ref{teo normale} and Corollary 3 in section \ref{Section normal bundle}.

Note that in the case $\dim X=1$ any smooth rational curve is a suitable $d$-Veronese embedding of $\mathbb{P}^{1}$, so the results of the present paper may be considered as a generalization of known results in the case of curves. In that case one can use the cohomology of twisted normal bundles  to determine their splitting type, but, when $\dim X\geq 2$, the dimensions of cohomology groups of the normal bundle in the various shifts appear as the only possible substitute for the splitting type, as the normal bundle itself may very well not be split into a sum of line bundles (it may even be stable), so there is no splitting type in general. Let us recall that, even for monomial maps $\PP^n\to \PP^s$, there is no general recipe for calculating the cohomology groups of twists of the normal bundle, for example the method of \cite{CoRi} for $n=1$ does not apply.

As mentioned above, we compute the full cohomology groups of the shifts of the normal bundle. The non-trivial part of this computation is provided by Theorem \ref{teo normale}. The result of Theorem \ref{teo normale} also makes explicit the upper semicontinuity of the cohomology of $\N_f(-k)$ with respect to $f$. This knowledge in principle could be applied to obtain a stratification of the space of maps $f$ with respect to the cohomology modules of $N_f$. This is not an easy task, indeed even in the case $n=1$ there is still no general description of the admissible Hilbert functions of the normal bundles to smooth rational curves. However, if one admits curves with at most ordinary nodes as singularities, such a description was provided in \cite{GhSa} and \cite{Sac}. Moreover it is known that the strata corresponding to a given Hilbert function of $\N_f$ may have many irreducible components, see for example the results of \cite{A-R1} or \cite{CoRi}. We leave the possible study of this stratification to future work.

When $f$ is not an embedding our results about the cohomology of twists of $\N_f$ provides information about the critical locus $Z$ of $f$, namely numerical bounds on the degree of the $1$-dimensional part of $Z$, if any, or on the degree of its $0$-dimensional part. In the case when $Z$ has both $0$-dimensional and $1$-dimensional components we show that their degrees are non-trivially bounded, by giving some examples and some general results linking these invariants to the cohomology of twists of $\N_f$. We consider these problems a promising line of research to pursue further in a future work.

The paper is organized as follows: in section \ref{sec:background} we set notations and we describe some representation-theoretic constructions involving an $(n+1)$-dimensional vector space $U$ and we prove some of the required properties of them. In section \ref{Section restricted tg} we describe how these constructions allow to compute the cohomology of $f^\ast \T_{\PP^s}$ (shortly $
\T_f$),  hence the cohomology of the restricetd tangent bundle of $X$ when $X$ is a smooth $d$-Veronese embedding. In section \ref{Section normal bundle} we deal with the cohomology of $\N_f$, hence with the cohomolgy of the normal bundle of $X$ when $X$ is a smooth $d$-Veronese embedding. In section \ref{Section n=2} we consider the case $n=2$ in detail, assuming that $X$ is a smooth $d$-Veronese embedding and  $f$ is a monomial map: in this case our method allows to calculate the cohomology by hand in many cases and that section is devoted to give some hints how to perform the calculation in the simplest cases. In the last section we consider maps $f$ which are not embeddings and we get the results about $Z$ quoted above.

\section{
Notation and preliminary results
}\label{sec:background}
\paragraph{\bf Notation}

$V : =$ finite dimensional vector space on $\mathbb{C}$

$ \langle  . . . \rangle  : =$ vector subspace generated by the elements
between the brackets

$V^{ \vee } : =Hom(V ,\mathbb{C})$ the dual vector space of $V$

$M^{t} : =$ transposed of the matrix $M$

$A^{ \bot } : =$ annihilator of $A :$ for any linear subspace $A \subseteq V ,$ when there exists a perfect symmetric pairing $ \langle $ , $ \rangle $$ :V \times V \rightarrow \mathbb{C} ,$ let us denote $A^{ \bot } =\{b \in V\vert  \langle a ,b \rangle  =0$  for any $a \in A\} .$

$Sec(Y) : =$ secant variety
of a projective variety $Y \subset \mathbb{P}^{t}$, the closure of the set of points belonging to all secant lines to $Y$

$Tan(Y) : =$ tangent variety
of a smooth projective variety $Y \subset \mathbb{P}^{t}$, the closure of the set of points belonging to all tangent spaces to $Y$

$V(I) :=$ variety associated to the ideal $I$

Let $U \simeq \mathbb{C}^{n +1}$ be a $(n +1)$-dimensional vector space and let $S^{d}U$ be the $d$-symmetric product of $U .$ Throughout this paper we will always assume $n \geq 1.$ Let $x_{0} ,x_{1} , . . . ,x_{n}$ be a base for $U$ and let $\underline{x} : =[x_{0} . . .x_{n}]$ be the corresponding matrix of vectors. Let us consider the linear operators $ \partial _{x_{0}} , \partial _{x_{1}} , . . . , \partial _{x_{n}}$ (in brief $ \partial _{0} , . . . , \partial _{n})$ acting on $U$ as partial derivatives. It is known that we can choose $ \langle  \partial _{0} , . . . , \partial _{n} \rangle $ as a dual base for $U^{ \vee }$ in such a way that every $\omega  \in U^{ \vee }$ can be written as $\omega  =\alpha _{0} \partial _{0} + . . . +\alpha _{n} \partial _{n}$ and the two bases $\{x_i\}$ and $\{\partial_j\}$ induce a perfect pairing $S^{d}U \times S^{d}U^{ \vee } \rightarrow \mathbb{C}$ for any $d \geq 1,$ defining by letting the second space act on the first by derivations. 

Let us consider a (non zero) subspace $T \subseteq S^{d}U$ with $d \geq 2.$ We can define

\begin{center}$ \partial T : = \langle \omega (T)\vert\  \omega  \in U^{ \vee } \rangle  = \langle  \partial _{0}T + . . . + \partial _{n}T \rangle  ,$\end{center}\par
\noindent for instance $ \partial S^{d}U =S^{d -1}U .$ We can also introduce the space

\begin{center}$ \partial ^{ -1}T : ={\displaystyle\bigcap _{\omega  \in U^{ \vee }}}[\omega ^{ -1}(T)] = \partial _{0}^{ -1}T \cap  . . . \cap  \partial _{n}^{ -1}0T$ , \end{center}\par
\noindent for instance $ \partial ^{ -1}S^{d}U =S^{d +1}U$. By using induction we can also define $ \partial ^{k}T$ and $ \partial ^{ -k}T$ for any $k \geq 2.$

For future use, let us recall some $GL(U)$-invariant operators acting between spaces of tensors on $U$ or $U^{ \vee } .$ Note that, if $y_{0} ,y_{1} , . . . ,y_{n}$ is another base for $U ,$ $\underline{y} : =[y_{0} . . .y_{n}]$ the corresponding matrix of vectors and $M$ is a non singular $(n +1 ,n +1)$ matrix representing the linear map $\underline{y} =M\underline{x}$, then the matrix acting on $\underline{ \partial } : =[ \partial _{0} . . . \partial _{n}] ,$  in order to preserve the perfect pairing, is $(M^{t})^{ -1} .$ In fact, if $\underline{v} =(M^{t})^{ -1}\underline{ \partial } ,$ then $ \langle \underline{v} ,\underline{y} \rangle  = \langle (M^{t})^{ -1}\underline{ \partial } ,M\underline{x} \rangle  =\underline{ \partial }^{t}M^{ -1}M\underline{x} =\underline{ \partial }^{t}\underline{x}$ is the identity $(n +1 ,n +1)$ matrix.  The previous pairs of transformations, acting on $U$ and $U^{ \vee }$ respectively, can be extended to $S^{d}U$ and to $S^{d}U^{ \vee }$ in a natural way and moreover to any tensor product $V$ of such spaces.

A linear operator $\phi  :V \rightarrow V^{ \prime }$, defined between vector spaces $V$ and $V^{ \prime }$ on which the group $GL(U)$ acts as above, is called $GL(U)$-linear if, for any $v \in V$ and for any $g \in GL(U) ,$ we have $\phi [g(v)] =g[\phi (v)] .$ A linear subspace $W \subset V$ is called $GL(U)$-invariant if, for any $w \in W$ and for any $g \in GL(U)$ we have: $g(w) \in W .$ For instance: if $\phi $ is $GL(U)$-linear then $\ker (\phi )$ and $\ensuremath{\operatorname*{Im}}(\phi )$ are $GL(U)$-invariant. Any subspace $W \subset V$ which is $GL(U)$-invariant defines a representation of $GL(U)$ which is a subrepresentation of $V .$

The following operators will be needed in the sequel.

1) General contractions. Due to the perfect pairing $S^{d}U \times S^{d}U^{ \vee } \rightarrow \mathbb{C}$ quoted above, we have contractions maps:

\begin{center}$S^{i}U^{ \vee } \otimes S^{j}U \rightarrow S^{j -i}U$ for $j \geq i \geq 0$\end{center}\par
\noindent such that every element of $S^{i}U^{ \vee }$ is thought as a degree $i$ polynomial of differential operators $ \partial _{0} , . . . , \partial _{n}$ acting on every degree $j$ polynomial of $S^{j}U$. For instance, if $i =j =d \geq 1 ,q \in S^{d}U^{ \vee }$and $l =a_{0}x_{0} + . . . +a_{n}x_{n} \in U$ we have that $q \otimes l^{d} \rightarrow q(l^{d})\  =d !q(a_{0} , . . . ,a_{n})$ .

\medskip

2) Multiplication maps. They are the $GL(U)$ linear maps

\begin{center}
$S^{i}U \otimes S^{j}U \rightarrow S^{i +j}U$ for $i \geq 0 ,j \geq 0$\end{center}\par

\noindent extending linearly the natural maps associating to any pure tensor $a \otimes b$ $ \in S^{i}U \otimes S^{j}U$ the product $ab \in S^{i +j}U$.

\medskip

3) Polarizations maps. They are $GL(U)$ linear maps

\begin{center}$p_{k} :S^{d +k}U \rightarrow S^{k}U \otimes S^{d}U$ for $k \geq 1 ,d \geq 1$\end{center}\par
\noindent which are proportional to the duals of the multiplication maps $m :S^{k}U^{ \vee } \otimes S^{d}U^{ \vee } \rightarrow S^{d +k}U^{ \vee }$ ; the proportionality factor is determined so
 that $m[p_{k}(q)] =q$ for any $q \in S^{d +k}U ,$ the polarizations maps are always injective and they are also uniquely defined by the condition $p_{k}(l^{d +k}) =l^{k} \otimes l^{d}$ for any $l \in U .$

In the sequel we will need an explicit way to write $p_{k}(q)$ for any $q \in S^{d +k}U .$ To this aim, let us introduce some notation. Let $I$ be a multi-index $(i_{0} ,i_{1} , . . . ,i_{n})$ where $i_{j}$ are non negative integers. Let us denote $\vert I\vert  : =\sum \limits _{j =0}^{n}i_{j} .$ By the symbol $x^{I}$ we will denote the monomial $x_{0}^{i_{0}}x_{1}^{i_{1}} . . .x_{n}^{i_{n}}$ of degree $\vert I\vert  .$ By the symbol $ \partial ^{I}f$ we will denote the partial derivative of any polynomial $f \in \mathbb{C}[x_{0} . . .x_{n}]$ by $ \partial _{0}^{i_{1}} \partial _{1}^{i_{1}} . . . \partial _{n}^{i_{n}}$. For any $r =0 , . . . ,n ,$ $I +1_{r} =(i_{0} ,i_{1} , . . . ,i_{r} +1 , . . . ,i_{n})$; analogously $I -1_{r} =(i_{0} ,i_{1} , . . . ,i_{r} -1 , . . . ,i_{n})$, of course only if $i_{r} \geq 1.$ Moreover $\binom{k}{I}$ $ : =\frac{k !}{i_{0} !i_{1} ! . . . .i_{n} !}$. 

Then we can say that (see \cite{A-R1}, formula (3.3)):

\begin{center}$p_{k}(q) =\frac{d !}{(d +k) !}\sum _{\vert I\vert  =k}\binom{k}{I}x^{I} \otimes  \partial ^{I}q .$\end{center}\par
\noindent Note that $\binom{k}{I}i_{r} =\binom{k}{I -1_{r}} .$ Note also that the sum with respect to $\vert I\vert  =k$ is the sum with respect to all monomials of degree $k$ in $n$ variables.

\medskip

4) The $GL(U)$ linear operator $\delta  :S^{k}U \otimes S^{d}U \rightarrow S^{k -1}U \otimes S^{d +1}U$ for any $k \geq 1$, $d \geq 0 ,$ such that:

\begin{center}$\delta (a \otimes b) =\sum \limits _{j =0}^{n} \partial _{j}(a) \otimes x_{j}b$ for any generator $a \otimes b \in S^{k}U \otimes S^{d}U .$\end{center}\par
\medskip

5)The $GL(U)$ linear operator $\theta  :S^{k}U \otimes S^{d}U \rightarrow S^{k +1}U \otimes S^{d -1}U$ for any $k \geq 0$, $d \geq 1 ,$ such that:

\begin{center}$\theta (a \otimes b) =\sum \limits _{j =0}^{n}x_{j}a \otimes  \partial _{j}(b$) for any generator $a \otimes b \in S^{k}U \otimes S^{d}U .$\end{center}\par

\medskip

6) $\psi _{k ,t}$ maps. They are maps

\begin{center}$\psi _{k ,t} :U \otimes S^{t}U \rightarrow U \otimes S^{k}U \otimes S^{t -k}U$$ \rightarrow S^{k}U \otimes S^{t -k +1}U$ for $k \geq 1 ,$  $t \geq k$\end{center}\par
\noindent which are the composition of $1 \otimes p_{k} :U \otimes S^{t}U \rightarrow U \otimes S^{k}U \otimes S^{t -k}U$ and the multiplication on the first and third factor: $U \otimes S^{t -k}U \rightarrow S^{t -k +1}U$. As $1 \otimes p_{k}$ and the multiplication are $GL(U)$ linear maps, we have that $\ensuremath{\operatorname*{Im}}(\psi _{k ,t})$ are $GL(U)$-invariant subspaces of $S^{k}U \otimes S^{t -k +1}U ,$ hence they define $GL(U)$ subrepresentations of this vector space.

\medskip

7) The operators $\xi _{i ,j} .$ For any pair of distinct integers $i ,j$ $ \in [0 ,n]$ with $i <j ,$ let us consider the element $x_{i} \otimes x_{j} -x_{j} \otimes x_{i} \in U \otimes U .$ Then,  we have a map

\begin{center}$\xi _{i ,j} :S^{k -1}U \otimes S^{d -1}U \rightarrow S^{k}U \otimes S^{d}U$ for $k \geq 2 ,$  $d \geq 2 ,$  \end{center}\par
\noindent given by the multiplication of every element of $S^{k -1}U \otimes S^{d -1}U$ by $x_{i} \otimes x_{j} -x_{j} \otimes x_{i} .$

\medskip

8) The operators $D_{i ,j} .$ For any pair of distinct integers $i ,j \in [0 ,n]$ with $i <j ,$ let us consider the element $ \partial _{i} \otimes  \partial _{j} - \partial _{j} \otimes  \partial _{i} \in U^{ \vee } \otimes U^{ \vee }$. Then, we have a map

\begin{center}$D_{i ,j} :S^{k}U \otimes S^{d}U \rightarrow S^{k -1}U \otimes S^{d -1}U$  for any $k \geq 1 ,d \geq 1$, \end{center}\par
\noindent given by the derivation of every element of $S^{k}U \otimes S^{d}U$ by $ \partial _{i} \otimes  \partial _{j} - \partial _{j} \otimes  \partial _{i} .$ Note that, for any $q \in S^{k +d}U ,$ we have $D_{r ,s}[p_{k}(q)] =0$. In fact $D_{r ,s}[p_{k}(q)] =D_{r ,s}[\frac{d !}{(d +k) !}\sum _{\vert I\vert  =k}\binom{k}{I}x^{I} \otimes  \partial ^{I}q] ,$ and forgetting the fixed coefficient $\frac{d !}{(d +k) !}$ we have:

$\sum _{\vert I\vert  =k;i_{r} \geq 1}\binom{k}{I -1_{r}}x^{I -1_{r}} \otimes  \partial ^{I +1_{s}}q -\sum _{\vert I\vert  =k;i_{s} \geq 1}\binom{k}{I -1_{s}}x^{I -1_{s}} \otimes  \partial ^{I +1_{r}}q =$

$\sum _{\vert H\vert  =k -1}\binom{k}{H}x^{H} \otimes  \partial ^{H +1_{r} +1_{s}}q -\sum _{\vert J\vert  =k -1}\binom{k}{J}x^{J} \otimes  \partial ^{J +1_{s} +1_{r}}q =0;$

\noindent for the last equalities we have put $H : =I -1_{r}$ when $i_{r} \geq 1$, and $J : =I -1_{s}$ when $i_{s} \geq 1.$

\medskip

We need to prove the following Lemma.

\begin{lemma}
\label{lemdelta}For any $q \in S^{d +k}U$ with $k \geq 0$ and $d \geq 1$ we have $\theta [p_{k}(q)] =(d)p_{k +1}(q) .$ For any $q \in S^{d +k}U$ with $k \geq 1$ and $d \geq 0$ we have $\delta [p_{k +1}(q)] =(k +1)p_{k}(q) .$ For any $q \in S^{d +k}U$ with $k \geq 1$ and $d \geq 1$ we have $\delta \{\theta [p_{k}(q)]\} =d(k +1)p_{k}(q) ,$ hence $\delta  \circ \theta $ is an automorphism of $p_{k}(S^{d +k}U) ,$ and $\theta \{\delta [p_{k +1}(q)]\} =d(k +1)p_{k +1}(q) ,$ hence $\theta  \circ \delta $ is an automorphism of $p_{k +1}(S^{d +k}U) .$ The operators $\delta $ and $\theta $ commute with all operators $\xi _{i ,j} .$
\end{lemma}

\begin{proof}
Let us write

$\theta [p_{k}(q)] =\theta [\frac{d !}{(d +k) !}\sum _{\vert I\vert  =k}\binom{k}{I}x^{I} \otimes  \partial ^{I}q] =\frac{d !}{(d +k) !}\sum _{\vert I\vert  =k}\binom{k}{I}\sum \limits _{j =0}^{n}x^{I +1_{j}} \otimes  \partial ^{I +1_{j}}q =$

$ =\frac{d !}{(d +k) !}\sum _{\vert H\vert  =k +1;h_{j} \geq 1}\sum \limits _{j =0}^{n}\binom{k}{H -1_{j}}x^{H} \otimes  \partial ^{H}q$;

\noindent for the last equality we have put $H : =I +1_{j}$, of course $h_{j} \geq 1.$

$(d)p_{k +1}(q) =\frac{d !}{(d +k) !}\sum _{\vert H\vert  =k +1}\binom{k +1}{H}x^{H} \otimes  \partial ^{H}q =$

$ =\frac{d !}{(d +k) !}\sum _{\vert H\vert  =k +1}\binom{k +1}{H}\sum \limits _{j =0}^{n}\frac{1}{k +1}x_{j} \partial _{j}(x^{H}) \otimes  \partial ^{H}q =$

$ =\frac{d !}{(d +k) !}\sum _{\vert H\vert  =k +1}\binom{k}{H}\sum \limits _{j =0}^{n}i_{j}x^{H} \otimes  \partial ^{H}q =$

$ =\frac{d !}{(d +k) !}\sum _{\vert H\vert  =k +1;h_{j} \geq 1}\sum \limits _{j =0}^{n}\binom{k}{H -1_{j}}x^{H} \otimes  \partial ^{H}q$

\noindent and we have done. 

We have used the Euler relation by writing $x^{H} =\frac{1}{k +1}\sum \limits _{j =0}^{n}x_{j} \partial _{j}(x^{H})$ because $\deg (x^{H}) =k +1.$

$\delta [p_{k +1}(q)] =\delta [\frac{(d -1) !}{(d +k) !}\sum _{\vert I\vert  =k +1}\binom{k +1}{I}x^{I} \otimes  \partial ^{I}q] =$

$ =\frac{(d -1) !}{(d +k) !}\sum _{\vert I\vert  =k +1}\binom{k +1}{I}\sum \limits _{j =0}^{n} \partial _{j}(x^{I}) \otimes x_{j} \partial ^{I}q =$

$ =\frac{(d -1) !}{(d +k) !}\sum _{\vert I\vert  =k +1;i_{j} \geq 1}\binom{k +1}{I}\sum \limits _{j =0}^{n}i_{j}x^{I -1_{j}} \otimes x_{j} \partial ^{I}q =$

$ =\frac{(d -1) !}{(d +k) !}\sum _{\vert I\vert  =k +1;i_{j} \geq 1}\sum \limits _{j =0}^{n}\binom{k +1}{I -1_{j}}x^{I -1_{j}} \otimes x_{j} \partial ^{I +1_{j} -1_{j}}q =$

$ =\frac{(d -1) !}{(d +k) !}\sum _{\vert H\vert  =k}\sum \limits _{j =0}^{n}\binom{k +1}{H}x^{H} \otimes x_{j} \partial ^{H +1_{j}}q$;

\noindent for the last equality we have put $H : =$$I -1_{j}$, of course when $i_{j} \geq 1.$

$(k +1)p_{k}(q) =(k +1)[\frac{d !}{(d +k) !}\sum _{\vert H\vert  =k}\binom{k}{H}x^{H} \otimes  \partial ^{H}q] =$

$ =\frac{d !}{(d +k) !}\sum _{\vert H\vert  =k}\binom{k +1}{H}x^{H} \otimes \frac{1}{d}\sum \limits _{j =0}^{n}x_{j} \partial ^{H +1_{j}}q] =$

$ =\frac{(d -1) !}{(d +k) !}\sum _{\vert H\vert  =k}\sum \limits _{j =0}^{n}\binom{k +1}{H}x^{H} \otimes x_{j} \partial ^{H +1_{j}}q$

\noindent and we have done. 

We have used the Euler relation by writing $ \partial ^{H}q =\frac{1}{d}\sum \limits _{j =0}^{n}x_{j} \partial ^{H +1_{j}}q$ because $\deg ( \partial ^{H}q) =d.$ 

The above relations show that $\delta \{\theta [p_{k}(q)]\} =d(k +1)p_{k}(q)$ and $\theta \{\delta [p_{k +1}(q)]\} =d(k +1)p_{k +1}(q) .$ The last property follows by straightforward calculations.
\end{proof}

Now we are ready to prove a particular Pieri decomposition for $S^{k}U \otimes S^{d}U$ as direct sum of irreducible $GL(U)$-representation (for which we refer to \cite{FH91}), for any $1 \leq k \leq d$.

\begin{proposition}
\label{propPieri} In the previous notation, for any $1 \leq k \leq d$, we have:
\end{proposition}

\begin{center}$S^{k}U \otimes S^{d}U =p_{k}(S^{d +k}U) \oplus \sum \xi _{i ,j}p_{k -1}(S^{d +k -2}U) \oplus $\end{center}\par
\begin{center}$ \oplus \sum \xi _{i ,j}\xi _{r ,s}p_{k -2}(S^{d +k -4}U) \oplus  . . . \oplus \sum \xi _{i ,j} . . .\xi _{r ,s}p_{0}(S^{d -k}U)$\end{center}\par
\noindent assuming that $S^{i}U =0$ for $i <0$, $S^{0}U =\mathbb{C}$ and that $p_{0}$ is the identity. Note that $\sum \xi _{i ,j}p_{k -1}(S^{d +k -2}U)$ is a short notation for the linear space

\begin{center}$ \langle \xi _{0 ,1}p_{k -1}(S^{d +k -2}U) ,\xi _{0 ,2}p_{k -1}(S^{d +k -2}U) , . . . ,\xi _{n -1 ,n}p_{k -1}(S^{d +k -2}U) \rangle $\end{center}\par
and so on.

\begin{proof}
First, let us prove the above Proposition when $n =1.$ In this case we have to prove that

\begin{center}$S^{k}U \otimes S^{d}U =p_{k}(S^{d +k}U) \oplus \xi _{0 ,1}p_{k -1}(S^{d +k -2}U) \oplus $\end{center}\par
\begin{center}$ \oplus \xi _{0 ,1}^{2}p_{k -2}(S^{d +k -4}U) \oplus  . . . \oplus \xi _{0 ,1}^{k}p_{0}(S^{d -k}U)$.\end{center}\par

For any $t =0 , . . . ,k ,$ we have that $\xi _{0 ,1}^{t}p_{k -t}(S^{d +k -2t}U)$ is a $GL(U)$-representation of $S^{k}U \otimes S^{d}U$ of dimension $d +k -2t +1.$ On the other hand, the standard Pieri decomposition of $S^{k}U \otimes S^{d}U$ is the direct sum of $k +1$ irreducible representations $\mathbb{S}_{(k +d ,0)} , . . . ,\mathbb{S}_{(d ,k)}$ such that $\dim (\mathbb{S}_{(k +d -t ,t)}) =d +k -2t +1$ (see \cite{FH91} pag. 81 and Theorem 6.3, in general it is not easy to calculate such dimensions, but it is very easy for $n =1$); therefore we can conclude that $\xi _{0 ,1}^{t}p_{k -t}(S^{d +k -2t}U) =\mathbb{S}_{(k +d -t ,t)}$ for any $t =0 , . . . ,k .$

From now on, let us assume that $n \geq 2$ and let us start with the case $d =k =1.$ It is well known that $U \otimes U =S^{2}U \oplus \bigwedge ^{2}U$ and that this is the decomposition of $U \otimes U$ as the sum of its irreducible $GL(U)$-representations.  On the other hand, $p_{1}(S^{2}U) \simeq S^{2}U$ is a $GL(U)$ invariant subspace of $U \otimes U ,$ as $p_{1}$ is $GL(U)$ linear, and a direct calculus shows that the same is true for $\sum \xi _{i ,j}p_{0}(S^{0}U) = \langle \xi _{0 ,1} ,\xi _{0 ,2} , . . . ,\xi _{n -1 ,n} \rangle $. Hence they are two $GL(U)$-representations of $U \otimes U$, a priori reducible, and they appear in the decomposition $U \otimes U =S^{2}U \oplus \bigwedge ^{2}U$. By calculating the dimension $\binom{n}{2}$ of $\sum \xi _{i ,j}$ it is immediate to see that $p_{1}(S^{2}U)$ $ \simeq S^{2}U$ and that $\sum \xi _{i ,j}$ $ \simeq \bigwedge ^{2}U .$

Now let us consider the cases with $k =1$ and $d \geq 2$. We know that $U \otimes S^{d}U =\mathbb{S}_{(1 +d ,0)}U \oplus \mathbb{S}_{(d ,1)}U$ and that $\mathbb{S}_{(1 +d ,0)}U =S^{d +1}U$ (see \cite{FH91} pag. 81); moreover, as in the case $d =1,$ it can be shown that $p_{1}(S^{d +1}U)$  and $\sum \xi _{i ,j}p_{0}(S^{d -1}U) =\sum \xi _{i ,j}S^{d -1}U$ are $GL(U)$-representations of $U \otimes S^{d}U$ (a priori reducible). Let us recall that the $p_{k}$ maps are always injective, hence $p_{1}(S^{d +1}U) \subseteq S^{d +1}U ,$ this latter intended as the irreducible summand of  $U \otimes S^{d}U =S^{d +1}U\oplus \mathbb{S}_{(d ,1)}U$. It follows that the $GL(U)$-representation $p_{1}(S^{d +1}U)$ is irreducible and it coincides with $\mathbb{S}_{(1 +d ,0)}U =S^{d +1}U .$ Moreover $p_{1}(S^{d +1}U) \cap \sum \xi _{i ,j}S^{d -1}U =0;$ in fact any element of $\sum \xi _{i ,j}S^{d -1}U \subseteq U \otimes S^{d}U$ is of the following type: $\sum _{i ,j =0 , . . . ,n}(x_{i} \otimes x_{j}f^{i ,j} -x_{j} \otimes x_{i}f^{i ,j})$ for suitable $f^{i ,j} \in S^{d -1}U ,$ hence $m[\sum _{i ,j =0 , . . . ,n}(x_{i} \otimes x_{j}f^{i ,j} -x_{j} \otimes x_{i}f^{i ,j}$$)] =0 ,$ where $m :U \otimes S^{d}U \rightarrow S^{d +1}U$ is the multiplication map. This fact proves that $\sum \xi _{i ,j}S^{d -1}U \subseteq \ker (m) .$ On the other hand, any element of $p_{1}(S^{d +1}U)$ is of the following type: $\frac{1}{d +1}(x_{0} \otimes  \partial _{0}q + . . . +x_{n} \otimes  \partial _{n}q)$ for a suitable $q \in S^{d +1}U .$ It follows that $m[\frac{1}{d +1}(x_{0} \otimes  \partial _{0}q + . . . +x_{n} \otimes  \partial _{n}q)] =q$, hence $\ker (m) \cap p_{1}(S^{d +1}$$U) =0$ and therefore $p_{1}(S^{d +1}U) \cap \sum \xi _{i ,j}S^{d -1}U =0.$ As $p_{1}(S^{d +1}U) =\mathbb{S}_{(1 +d ,0)}U$ we get $\sum \xi _{i ,j}S^{d -1}U =\mathbb{S}_{(d ,1)}U$.

Now let us consider the general cases with $k \geq 2$ and $d \geq 2$ and let us proceed by induction on $k .$ Let us fix any $d \geq 2;$ by induction, for any $k \geq 2 ,$ we can assume that

\begin{center}$S^{k -1}U \otimes S^{d +1}U$ $ =p_{k -1}(S^{d +k}U) \oplus \sum \xi _{i ,j}p_{k -2}(S^{d +k -2}U) \oplus $\end{center}\par
\begin{center}$ \oplus \sum \xi _{i ,j}\xi _{r ,s}p_{k -3}(S^{d +k -6}U) \oplus  . . . \oplus \sum \xi _{i ,j} . . .\xi _{r ,s}p_{0}(S^{d -k +2}U) .$\end{center}\par
\noindent and we want to prove that

\begin{center}$S^{k}U \otimes S^{d}U$ $ =p_{k}(S^{d +k}U) \oplus \sum \xi _{i ,j}p_{k -1}(S^{d +k -2}U) \oplus $\end{center}\par
\begin{center}$ \oplus \sum \xi _{i ,j}\xi _{r ,s}p_{k -2}(S^{d +k -4}U) \oplus  . . . \oplus \sum \xi _{i ,j} . . .\xi _{r ,s}p_{0}(S^{d -k}U) .$\end{center}\par
To simplify notations, let us write $A_{t} : =\sum \xi _{i ,j} . . .\xi _{r ,s}p_{k -t}(S^{d +k -2t}U)$ for $t =0 , . . . ,k$ and $B_{t} : =\sum \xi _{i ,j} . .\xi _{r ,s}p_{k -1 -t}(S^{d +k -2t}U)$ for $t =0 , . . . ,k -1.$ So that we can assume that

\begin{center}$S^{k -1}U \otimes S^{d +1}U =B_{0} \oplus B_{1} \oplus  . . . \oplus B_{k -1}$\end{center}\par
\noindent and we have to prove that

\begin{center} $S^{k}U \otimes S^{d}U =A_{0} \oplus A_{1} \oplus  . . . \oplus A_{k -1} \oplus A_{k} .$\end{center}\par
Note that every $A_{t}$ is a $GL(U)$-representation of $S^{k}U \otimes S^{d}U ,$ a priori reducible, and we know that the standard Pieri decomposition of $S^{k}U \otimes S^{d}U$ as direct sum of irreducible $GL(U)$-representations is :

\begin{center}$S^{k}U \otimes S^{d}U =\mathbb{S}_{(k +d ,0)}U \oplus \mathbb{S}_{(k +d -1 ,1)}U \oplus  . . . \oplus \mathbb{S}_{(d ,k)}U$\end{center}\par
\noindent (see \cite{FH91} pag. 81) hence, to complete our proof, it suffices to show that $\{A_{0} , . . . ,A_{k}\} =\{\mathbb{S}_{(k +d ,0)}U , . . . ,\mathbb{S}_{(d ,k)}U\}$ as sets of $k +1$ elements. A priori, every $A_{t}$ is the direct sum of some elements of $\{\mathbb{S}_{(k +d ,0)}U , . . . ,\mathbb{S}_{(d ,k)}U\}$; to complete our proof it suffices to show that every $A_{t}$ coincides exactly with one element of the set $\{\mathbb{S}_{(k +d ,0)}U , . . . ,\mathbb{S}_{(d ,k)}U\}$ and we can assume, by induction, that $\{B_{0} , . . . ,B_{k -1}\} =\{\mathbb{S}_{(d +k -1 ,0)}U , . . . ,\mathbb{S}_{(d ,k -1)}U\}$ as sets of $k$ elements. 

Let us proceed by contradiction: if the previous statement is false, then there exists at least a non zero direct sum of some elements of $\{\mathbb{S}_{(k +d ,0)}U , . . . ,\mathbb{S}_{(d ,k)}U\}$, say $S ,$ and two distinct integers $\alpha  ,\beta  \in [0 ,k]$ such that $A_{\alpha } =S \oplus S_{\alpha }$ and $A_{\beta } =S \oplus S_{\beta }$ where $S_{\alpha }$ and $S_{\beta }$ are other direct sums of elements of $\{\mathbb{S}_{(k +d ,0)}U , . . . ,\mathbb{S}_{(d ,k)}U\}$ with $S_{\alpha } \cap S_{\beta } =0$ (possibly $S_{\alpha } =0 ,$ and/or $S_{\beta } =0$).

By Lemma \ref{lemdelta} we know that, for any $t =0 , . . . ,k -1$, we have: $\theta (B_{t}) \subseteq A_{t}$; $\delta (A_{t}) \subseteq B_{t}$ $;$ $\delta [\theta (B_{t})] \simeq B_{t}$  and $\delta (A_{k}) =0.$ Of course $\delta (A_{t})$ is a $GL(U)$-representation of $S^{k -1}U \otimes S^{d +1}U ,$ by Lemma \ref{lemdelta} it follows that it is non zero and it is contained in $B_{t}$ for any $t =0 , . . . ,k -1$. By induction assumption, every $B_{t}$ is an irreducible representation of $S^{k -1}U \otimes S^{d +1}U ,$ then we can conclude that $\delta (A_{t}) =B_{t}$ for any $t =0 , . . . ,k -1$, while $\delta (A_{k}) =0.$

Now let us apply $\delta $ to $A_{\alpha }$ and $A_{\beta } .$ If $\alpha  \leq k -1$ and $\beta  \leq k -1$ we get $B_{\alpha } =\delta (A_{\alpha }) =\delta (S \oplus S_{\alpha })$, hence $\delta (S) \subseteq B_{\alpha }$ and $\delta (S_{\alpha }) \subseteq B_{\alpha };$ analogously $B_{\beta } =\delta (A_{\beta }) =\delta (S \oplus S_{\beta })$, hence $\delta (S) \subseteq B_{\beta }$ and $\delta (S_{\beta }) \subseteq B_{\beta } .$ As $\delta (S) ,$ $\delta (S_{\alpha }) ,$ $\delta (S_{\beta })$ are $GL(U)$-representations of $S^{k -1}U \otimes S^{d +1}U$ and $B_{\alpha } \cap B_{\beta } =0$ the unique possibility is $\delta (S_{\alpha }) =B_{\alpha } ,$ $\delta (S_{\beta }) =B_{\beta } ,$ $\delta (S) =0.$ If $\alpha  \leq k -1$ and $\beta  =k$ we have $0 =\delta (A_{k}) =\delta (S \oplus S_{k}) ,$ hence $\delta (S) =0.$ The same argument runs if $\alpha  =k$ and $\beta  \leq k -1.$ In any case $\delta (S) =0$. This means that for any pair $A_{\alpha } =S \oplus S_{\alpha }$ and $A_{\beta } =S \oplus S_{\beta }$ as above, we have $\delta (S) =0.$ Hence $S \neq 0$ is the direct sum of some elements in $\{\mathbb{S}_{(k +d ,0)}U , . . . ,\mathbb{S}_{(d ,k)}U\}$ all contained in $\ker (\delta ) .$

 Therefore we have the following conditions:

- every $A_{t}$ is the direct sum of some elements of $\{\mathbb{S}_{(k +d ,0)}U , . . . ,\mathbb{S}_{(d ,k)}U\}$ for $t =0 , . . . ,k$;

- $\delta (A_{t}) =B_{t} \neq 0$ for $t =0 , . . . ,k -1$ and $\delta (A_{k}) =0$;

- at least one element in $\{\mathbb{S}_{(k +d ,0)}U , . . . ,\mathbb{S}_{(d ,k)}U\}$ is contained in $\ker (\delta )$;

\noindent but there are no sufficient elements in $\{\mathbb{S}_{(k +d ,0)}U , . . . ,\mathbb{S}_{(d ,k)}U\}$ to guarantee the above conditions, unless $S$ $ \supseteq $ $A_{k}$ (for instance, it could be $k =3 ,$ $A_{0} =\mathbb{S}_{(3 +d ,0)}U ,$ $A_{1} =\mathbb{S}_{(2 +d ,1)}U$, $A_{2} =S_{(1 +d ,2)}U \oplus S_{(d ,3)}U ,$ $A_{3} =S_{(d ,3)}U$).

However $S$ cannot contain $A_{k}$, otherwise $A_{k}$ would be also contained in some $A_{\alpha }$ with $\alpha  \in [0 ,k -1]$ and this is not possible, otherwise the same thing would be true also for $n =1 ,$ but we know that $A_{0} \oplus A_{1} \oplus  . . . \oplus A_{k}$ is a direct sum when $n =1,$ as a consequence of the fact that the $A_i$'s are distinct irreducible representations of $GL(U)$.
\end{proof}

Note that we can write the above decomposition in a different way in order to obtain the following two relations:

$S^{k -1}U \otimes S^{d -1}U =p_{k -1}(S^{d +k -2}U) \oplus {\displaystyle\sum _{i ,j =0...n}}\xi _{i ,j}(S^{k -2}U \otimes S^{d -2}U)$

$S^{k}U \otimes S^{d}U =p_{k}(S^{d +k}U) \oplus {\displaystyle\sum _{i ,j =0...n}}\xi _{i ,j}(S^{k -1}U \otimes S^{d -1}U) =$

$ =p_{k}(S^{d +k}U) \oplus {\displaystyle\sum _{i ,j =0...n}}\xi _{i ,j}p_{k -1}(S^{d +k -2}U) \oplus {\displaystyle\sum _{i ,j =0...n;r ,s =0...n}}\xi _{i ,j}\xi _{r ,s}(S^{k -2}U \otimes S^{d -2}U)$.

The following Propositions give some relations between the previously introduced operators.

\begin{proposition}
\label{propinter}In the previous notation, for any $1 \leq k \leq d$ we have:
\end{proposition}

$\operatorname{1)\,p}_{k}(S^{d +k}U) ={\displaystyle\bigcap _{i ,j =0...n}}\ker (D_{i ,j})$

$\operatorname{2)\,p}_{k}(S^{d +k}U) \oplus {\displaystyle\sum _{i ,j =0...n}}\xi _{i ,j}p_{k -1}(S^{d +k -2}U) ={\displaystyle\bigcap _{i ,j =0...n;r ,s =0...n}}\ker (D_{i ,j} \circ D_{r ,s})$.

Before proving Proposition \ref{propinter} we need the following two Lemmas.

\begin{lemma}
\label{lemmapk}Let $f \in S^{t +\varepsilon }U$ with $t \geq 1$ and $\varepsilon  \geq 1$. Let us consider $p_{t}(f) =\frac{\varepsilon  !}{(\varepsilon  +t) !}\sum _{\vert I\vert  =t}\binom{t}{I}x^{I} \otimes  \partial ^{I}f \in S^{t}U \otimes S^{\varepsilon }U;$ if we transform any element $a \otimes b$ of the sum with respect to $\vert I\vert  =t$ in the following way, for a fixed pair of integers $(p,q) \in [0 ,n]$:

\begin{center}$a \otimes b \rightarrow $ $a \otimes x_{p} \partial _{q}b +x_{p} \partial _{q}a \otimes b$\end{center}\par
\noindent then we get the element $p_{t}(x_{p} \partial _{q}f) \in S^{t}U \otimes S^{\varepsilon }U .$
\end{lemma}

\begin{proof}
Let us apply the above transformation to any addend of $p_{t}(f)$ and, by forgetting the coefficient $\frac{\varepsilon  !}{(\varepsilon  +t) !}$ , we get:

$\sum _{\vert I\vert  =t}\binom{t}{I}x^{I} \otimes (x_{p} \partial ^{I +1_{q}}f) +\sum _{\vert I\vert  =t;i_{q} \geq 1}\binom{t}{I -1_{q}}x^{I +1_{p} -1_{q}} \otimes  \partial ^{I}f =$

$\sum _{\vert I\vert  =t}\binom{t}{I}x^{I} \otimes (x_{p} \partial ^{I +1_{q}}f) +\sum _{\vert J\vert  =t -1}\binom{t -1}{J}x^{J +1_{p}} \otimes  \partial ^{J +1_{q}}f .$

On the other hand, if we determine $p_{t}(x_{p} \partial _{q}f)$ and we forget the same coefficient $\frac{\varepsilon  !}{(\varepsilon  +t) !}$, we have:

$\sum _{\vert I\vert  =t}\binom{t}{I}x^{I} \otimes  \partial ^{I}(x_{p} \partial _{q}f)$$ =$

$ =\sum _{\vert I\vert  =t;i_{p} \geq 1}\binom{t}{I}x^{I} \otimes (i_{p} \partial ^{I -1_{p} +1_{q}}f +x_{p} \partial ^{I +1_{q}}f) +\sum _{\vert I\vert  =t;i_{p} =0}\binom{t}{I}x^{I} \otimes x_{p} \partial ^{I +1_{q}}f =$

$ =\sum _{\vert I\vert  =t}\binom{t}{I}x^{I} \otimes x_{p} \partial ^{I +1_{q}}f +\sum _{\vert I\vert  =t;i_{p} \geq 1}\binom{t}{I -1_{p}}x^{I} \otimes  \partial ^{I -1_{p} +1_{q}}f =$

$ =\sum _{\vert I\vert  =t}\binom{t}{I}x^{I} \otimes x_{p} \partial ^{I +1_{q}}f +\sum _{\vert J\vert  =t -1}\binom{t -1}{J}x^{J +1_{p}} \otimes  \partial ^{J +1_{q}}f$

which is exactly the previous element in $S^{t}U \otimes S^{\varepsilon }U .$
\end{proof}

\begin{lemma}
\label{leminters kern} Let us consider the operators $D_{i ,j} :S^{k}U \otimes S^{d}U \rightarrow S^{k -1}U \otimes S^{d -1}U$ for any  $0 \leq i <j \leq n$ with $k \geq 2 ,$ $d \geq 2.$ 

Then ${\displaystyle\bigcap _{i ,j =0...n}}\ker (D_{i ,j})$ and ${\displaystyle\bigcap _{i ,j =0...n;r ,s =0...n}}\ker (D_{i ,j} \circ D_{r ,s})$ are subspaces $GL(U)$-invariant.
\end{lemma}

\begin{proof}
 Note that $K : ={\displaystyle\bigcap _{i ,j =0...n}}\ker (D_{i ,j})$ can be described as follows:

$K =\{p \otimes q \in S^{k}U \otimes S^{d}U\vert $ $\omega (p \otimes q) =0$ for any $\omega  \in \bigwedge ^{2}U^{ \vee }\}$.

For any $g \in GL(U)$, let $g_{1} ,g_{2} ,g_{3}$ be the induced actions on $S^{k}U \otimes S^{d}U ,$ $S^{k -1}U \otimes S^{d -1}U$ and $\bigwedge ^{2}U^{ \vee } ,$ respectively. We know that, for any $p \otimes q \in S^{k}U \otimes S^{d}U ,$ for any $g \in SL(U)$ and for any $\omega  \in \bigwedge ^{2}U^{ \vee } ,$ we have: $g_{2}[\omega (p \otimes q)] =g_{3}(\omega )[g_{1}(p \otimes q)] .$ Hence, if $p \otimes q \in K ,$ we have $g_{3}(\omega )[g_{1}(p \otimes q)] =0$. 

Let $p \otimes q$ be any element in $K$ and let $g \in GL(U);$ we have to show that $g_{1}(p \otimes q) \in K ,$ i.e. that, for any $\omega  \in \bigwedge ^{2}U^{ \vee }$, we have $\omega [g_{1}(p \otimes q)] =0.$ It suffices to remark that there is a suitable $\omega ^{ \prime } \in \bigwedge ^{2}U^{ \vee }$ such that $\omega  =g_{3}(\omega ^{ \prime })$ (it suffices to choose $\omega ^{ \prime } =g_{3}^{ -1}(\omega )$); hence $\omega [g_{1}(p \otimes q)] =g_{3}(\omega ^{ \prime })[g_{1}(p \otimes q)] =0$.

For the second part of the Lemma we can argue in the same way. Note that $\ker (D_{i ,j})$ is not $GL(U)$-invariant for a fixed pair $i <j$ when $n \geq 2.$
\end{proof}

\begin{proof}
(of Proposition \ref{propinter}) Let us recall that Proposition \ref{propinter} is true when $n =1 :$ see Corollary 3.2 of \cite{A-R1} ; let us assume $n \geq 2.$

1) By the definition of $D_{i ,j}$ it is obvious that $p_{k}(S^{d +k}U) \subseteq {\displaystyle\bigcap _{i ,j =1...n}}\ker (D_{i ,j})$. On the other hand ${\displaystyle\bigcap _{i ,j =0...n}}\ker (D_{i ,j})$ is a $GL(U)$-representation of $S^{k}U \otimes S^{d}U$ by Lemma \ref{leminters kern}, hence, by Proposition \ref{propPieri}, it is the direct sum of $p_{k}(S^{d +k}U)$ and some vector space $A_{t} : =\sum \xi _{i ,j} . . .\xi _{r ,s}p_{k -t}(S^{d +k -2t}U)$ for $t =1 , . . . ,k$. This fact would imply that $A_{t}$ $ \subseteq \ker (D_{0 ,1})$ for some $t \in [1 ,k] ,$ also in case $n =1 ,$ as one can see by setting equal to $0$ all the variables except $x_0,x_1$. This is not possible because Proposition \ref{propinter} is true when $n =1.$

2) Let us consider, more generically, the action of any operator $D_{r ,s}$$ \circ \xi _{i ,j}$ acting on $a \otimes b \in S^{p}U \otimes S^{q}U .^{}$ Let $a_{r}$ be $ \partial _{r}(a) ,$ let $a_{s}$ be $ \partial _{s}(a)$ and so on. Then we have:

$D_{r ,s}[\xi _{i ,j}(a \otimes b)] =D_{r ,s}(x_{i}a \otimes x_{j}b -x_{j}a \otimes x_{i}b) =$

$ =2[ \partial _{r}(x_{i}) \partial _{s}(x_{j}) - \partial _{s}(x_{i}) \partial _{r}(x_{j})]a \otimes b +$

$ +a \otimes [( \partial _{s}(x_{j})x_{i} - \partial _{s}(x_{i})x_{j})b_{r} +( \partial _{r}(x_{i})x_{j} - \partial _{r}(x_{j})x_{i})b_{s}] +$

$ +[( \partial _{s}(x_{j})x_{i} - \partial _{s}(x_{i})x_{j})a_{r} +( \partial _{r}(x_{i})x_{j} - \partial _{r}(x_{j})x_{i})a_{s}] \otimes b +\xi _{i ,j}[D_{r ,s}(a \otimes b)] .$

Now, by recalling that $r <s$ and $i <j ,$ we can distinguish six cases:

\medskip

$i$$)$ $i \neq r ,$ $j \neq s ,$ $i \neq s ,$ $j \neq r ,$ then

$D_{r ,s}[\xi _{i ,j}(a \otimes b)] =\xi _{i ,j}[D_{r ,s}(a \otimes b)]$

\medskip

$ii)$ $i =r ,$ $j =s ,$ $i \neq s ,$ $j \neq r ,$ then 

$D_{r ,s}[\xi _{i ,j}(a \otimes b)] =2a \otimes b +a \otimes (x_{i}b_{r} +x_{j}b_{s}) +(x_{i}a_{r} +x_{j}a_{s}) \otimes b +\xi _{i ,j}[D_{r ,s}(a \otimes b)]$

\medskip

$iii$$)$ $i =r ,$ $j \neq s ,$ $i \neq s ,$ $j \neq r ,$ then

$D_{r ,s}[\xi _{i ,j}(a \otimes b)] =a \otimes x_{j}b_{s} +x_{j}a_{s} \otimes b +\xi _{i ,j}[D_{r ,s}(a \otimes b)]$

\medskip

$iv)$ $i \neq r ,$ $i =s ,$ $j \neq s ,$ $j \neq r ,$ then

$D_{r ,s}[\xi _{i ,j}(a \otimes b)] = -a \otimes x_{j}b_{r} -x_{j}a_{r} \otimes b +\xi _{i ,j}[D_{r ,s}(a \otimes b)]$

\medskip

$v)$ $i \neq r$, $i \neq s ,$ $j \neq s ,$ $j =r ,$ then

$D_{r ,s}[\xi _{i ,j}(a \otimes b)] = -a \otimes x_{i}b_{s} -x_{i}a_{s} \otimes b +\xi _{i ,j}[D_{r ,s}(a \otimes b)]$

\medskip

$vi)$ $i \neq r ,$ $i \neq s ,$ $j =s ,$ $j \neq r ,$ then

$D_{r ,s}[\xi _{i ,j}(a \otimes b)] =a \otimes x_{i}b_{r} +x_{i}a_{r} \otimes b +\xi _{i ,j}[D_{r ,s}(a \otimes b)]$.

\medskip

Now, let us apply $D_{r ,s} \circ \xi _{i ,j}$ to the generic element $p_{k -1}(f) \in p_{k -1}(S^{d +k -2}U)$,  i.e. let $f$ be any polynomial in $S^{d +k -2}U .$ Let us use Lemma \ref{lemmapk} and let us recall that $D_{r ,s}(p_{k -1}(f)) =0$ for any $0 \leq r <s \leq n$  and $k \geq 1.$ We get:

$i)$ $D_{r ,s}[\xi _{i ,j}(p_{k -1}(f))] =0$

$ii)$ $D_{r ,s}[\xi _{i ,j}(p_{k -1}(f))] =2p_{k -1}(f) +p_{k -1}(x_{i}f_{r})$ $ +p_{k -1}(x_{j}f_{s})$

$iii)$ $D_{r ,s}[\xi _{i ,j}(p_{k -1}$$(f)] =p_{k -1}(x_{j}f_{s})$

$iv)$ $D_{r ,s}[\xi _{i ,j}(p_{k -1}(f)] = -p_{k -1}(x_{j}f_{r})$

$v)$ $D_{r ,s}[\xi _{i ,j}(p_{k -1}(f)] = -p_{k -1}(x_{i}f_{s})$

$vi)$ $D_{r ,s}[\xi _{i .j}(p_{k -1}(f)] =p_{k -1}(x_{i}f_{r})$ .

It follows that, in any case, $D_{u ,v}\{D_{r ,s}[\xi _{i ,j}(p_{k -1}(f)]\} =0$ for any pair of integers $u <v ,$ $r <s ,$ $i <j$ and this fact proves that

 $p_{k}(S^{d +k}U) \oplus {\displaystyle\sum _{i ,j =0...n}}\xi _{i ,j}p_{k -1}(S^{d +k -2}U) \subseteq {\displaystyle\bigcap _{i ,j =0...n;r ,s =0...n}}\ker (D_{i ,j} \circ D_{r ,s})$.

For the other inclusion we can argue as in item 1).

\end{proof}

\begin{proposition}
\label{propIm-ker}Let us consider the maps $\psi _{k ,d +k -1} :U \otimes S^{d +k -1}U \rightarrow S^{k}U \otimes S^{d}U$ for any $k \geq 1$ and $d \geq k \geq 1$ then we have:

\begin{center}$\ensuremath{\operatorname*{Im}}(\psi _{k ,d +k -1}) ={\displaystyle\bigcap _{i ,j =0...n;r ,s =0...n}}\ker (D_{i ,j} \circ D_{r ,s})$.\end{center}\par
\end{proposition}

\begin{proof}
According to Pieri decompositions, we have:

\begin{center}$U \otimes S^{d +k -1}U =p_{1}(S^{d +k}U) \oplus {\displaystyle\sum _{i ,j =0...n}}\xi _{i ,j}(S^{d +k -2}U)$\end{center}\par

\begin{center}$S^{k}U \otimes S^{d}U =p_{k}(S^{d +k}U) \oplus {\displaystyle\sum _{i ,j =0...n}}\xi _{i ,j}p_{k -1}(S^{d +k -2}U) \oplus ...$\end{center}\par
If we restrict $\psi _{k ,d +k -1}$ to the two components of the Pieri decomposition of $U \otimes S^{d +k -1}U$ we see immediately that

\begin{center}$0 \neq \psi _{k ,d +k -1}[p_{1}(S^{d +k}U)] \subseteq p_{k}(S^{d +k}U)$\end{center}\par
\begin{center}$0 \neq \text{}\psi _{k ,d +k -1}[{\displaystyle\sum _{i ,j =0...n}}\xi _{i ,j}(S^{d +k -2}U)] \subseteq {\displaystyle\sum _{i ,j =0...n}}\xi _{i ,j}p_{k -1}(S^{d +k -2}U)$\end{center}\par
\noindent moreover, as $Im(\psi _{k ,d +k -1})$  is a subspace $GL(U)$-invariant giving rise to a $GL(U)$-representation of $U \otimes S^{d +k -1}U$ and the components of the Pieri decomposition are all the irreducible $GL(U)$-representations of the vector spaces $S^{a}U \otimes S^{b}U$, we get that

\begin{center}$\psi _{k ,d +k -1}[p_{1}(S^{d +k}U)] =p_{k}(S^{d +k}U)$\end{center}\par
\begin{center}$\psi _{k ,d +k -1}[{\displaystyle\sum _{i ,j =0...n}}\xi _{i ,j}(S^{d +k -2}U)] ={\displaystyle\sum _{i ,j =0...n}}\xi _{i ,j}p_{k -1}(S^{d +k -2}U)$.\end{center}\par
Hence, by Proposition \ref{propinter}, we get:

\begin{center}$\ensuremath{\operatorname*{Im}}(\psi _{k ,d +k -1}) ={\displaystyle\bigcap _{i ,j =0...n;r ,s =0...n}}\ker (D_{i ,j} \circ D_{r ,s})$. 
\end{center}\par
\end{proof}

\section{Rational varieties as projections of Veronese embeddings and cohomology of the restricted tangent bundle  
}\label{Section restricted tg}
\label{Section tangent bundle}
Given any $\mathbb{C}$-vector space $W$, let us denote by $\mathbb{P}(W)$ the projective space of $1$-dimensional subspaces of $W .$ If $E \subset W$ is a $(e +1)$-dimensional subspace of $W ,$ let us denote by $\mathbb{P}(E)$ the corresponding projective subspace; if $w \in W$ is a non zero vector, $[w]$ will be the associated point in $\mathbb{P}(W) .$ 

As in section 2, let $U \simeq \mathbb{C}^{n +1}$ be a $(n +1)$-dimensional vector space and $\mathbb{P}^{n} =\mathbb{P}(U)$ the associated projective space. Let us assume $d \geq 2$ and let $S^{d}$$U$  be the $d$-symmetric product of $U$ and let $\nu _{d} :\mathbb{P}^{n} \rightarrow \mathbb{P}(S^{d}U) =\mathbb{P}^{N}$ be the usual $d$-embedding of Veronese; of course $N =\binom{n +d}{d} -1.$

Let us consider a map $f :\mathbb{P}(U) \rightarrow \mathbb{P}^{s}$, defined by $s+1$ homogeneous polynomials of degree $d$. Denoting $X=f(\mathbb{P}(U))$, then $X$ is always the projection in $\mathbb{P}^{s}$ of $\nu _{d}(\mathbb{P}^{n}) \subset \mathbb{P}^{N}$ by a suitable projective subspace $\mathbb{P}(T)$ such that $\dim (T) = e +1 = N - s$, where $\mathbb{P}^{s} : =\mathbb{P}(V)$, $n \geq 1$. We have that $f$ is in fact defined by an injection

\begin{center}$f^{ \ast } :H^{0}(\mathbb{P}^{s} ,\mathcal{O}_{\mathbb{P}^{s}}(1)) =V^{ \vee } \hookrightarrow H^{0}(\mathbb{P}^{n} ,\mathcal{O}_{\mathbb{P}^{n}}(d)) =S^{d}U^{ \vee }$\end{center}\par

\noindent such that the polynomials in  $f^{ \ast }(V^{ \vee })$ (a basis of which gives the polynomials defining $f$) have no common roots. We have $T =f^{ \ast }(V^{ \vee })^{ \bot } \subset S^{d}U .$ Then one sees that $f^{ \ast }$ can be identified with the dual of the natural projection $S^{d}U \rightarrow S^{d}U/T \simeq V$, hence $\mathbb{P}(V) \simeq \mathbb{P}($$S^{d}U/T)$ and $f$ can be identified with $\pi _{T} \circ \nu _{d}$ where $\pi _{T} :\mathbb{P}^{N} \rightarrow \mathbb{P}^{s}$ is the projection with vertex $\mathbb{P}(T) .$ Note that $\dim (T^{ \vee }) =s +1 =N - e $.  Moreover the natural inclusion $(S^{d}U/T)^{ \vee } \subset S^{d}U^{ \vee }$ identifies $(S^{d}U/T)^{ \vee }$ with $T^{\perp }$, hence $S^{d}U/T \simeq (T^{\perp })^{ \vee } .$

We have the following fundamental exact sequence:

\begin{equation}\label{eq:exactTN}
0 \rightarrow \mathcal{T}_{\mathbb{P}^{n}} \rightarrow f^{ \ast }\mathcal{T}_{\mathbb{P}^{s}} \rightarrow \mathcal{N}_{f} \rightarrow 0
\end{equation}
 where  $f^{ \ast }\mathcal{T}_{\mathbb{P}^{s}}$ is the pull back of the tangent bundle of $\mathbb{P}^{s}$, the first map is the differential of $f$ and $\mathcal{N}_{f }$ is its cokernel. Similarly, the pull back $f^{ \ast }\mathcal{T}_{\mathbb{P}^{s}}$ of the tangent bundle of $\mathbb{P}^{s}$ will be often denoted by $\mathcal{T}_{f}$.

\medskip 

\begin{remark} The fact that $f:\PP^n=\PP(U)\to \PP^s=\PP(S^dU/T)$ is everywhere defined, and the identification $f=\pi_T\circ\nu_d$ discussed above, amounts to the condition
$$\mathbb{P}(T) \cap \nu _{d}(\mathbb{P}(U) = \varnothing.$$
\end{remark}

\medskip 

\begin{remark}
If moreover we assume that $f$ is an embedding, identifying $\PP^n$ with its image $X$, which is a smooth rational, projective, $n$-dimensional variety embedded in $\mathbb{P}^{s}$, then $$\mathbb{P}(T) \cap Sec(\nu _{d}(\mathbb{P}(U)) = \varnothing.$$
Note that in particular one must have $s \geq 2n +1$.

\medskip 

 Let $\mathcal{I}_{X}$ be the ideal sheaf of $X$ in $\mathbb{P}^{s}$. The normal bundle of $X$ in $\mathbb{P}^{s}$ is defined as $\mathcal{N}_{X} : =Hom(\mathcal{I}_{X}/\mathcal{I}_{X}^{2} ,$$\mathcal{O}_{X})$.  When $f$ is an embedding, $\T_f$ can be identified with the restriction of $\mathcal{T}_{\mathbb{P}^{s}}$ to $X $. Moreover $f$ identifies the pull back $f^{ \ast }\mathcal{N}_{X}$ of the normal bundle with $\mathcal{N}_{X}$ and we also have $f^{ \ast }\mathcal{N}_{X} = \mathcal{N}_{f}$. Hence, in our assunptions, the sequence (\ref{eq:exactTN}) becomes:

\begin{equation}\label{eq:exactTNembedding}
0 \rightarrow \mathcal{T}_{\mathbb{P}^{n}} \rightarrow \mathcal{T}_{f} \rightarrow \mathcal{N}_{X} \rightarrow 0.
\end{equation}
\end{remark}

Now let us go back to the general case, where $f$ is just a map.
Taking into account the previous notation and considering the Euler sequence for $\mathbb{P}^{n}$ and the pull back of the Euler sequence for $\mathbb{P}^{s}$ we get the following commutative diagram:

\begin{equation}\label{eq:diagramnormal}
\begin{array}{ccccccccc}\, & \, & 0 & \, & 0 & \, & \, & \, & \, \\
\, & \, & \downarrow  & \, & \downarrow  & \, & \, & \, & \, \\
0 &  \rightarrow  & \mathcal{O}_{\mathbb{P}^{n}}\mathbb{}^{} &  \rightarrow  & \mathcal{O}_{\mathbb{P}^{n}}\mathbb{}^{} &  \rightarrow  & 0 & \, & \, \\
\, & \, & \downarrow  & \, & \downarrow  & \, & \downarrow  & \, & \, \\
0 &  \rightarrow  & U \otimes \mathcal{O}_{\mathbb{P}^{n}}(1) &  \rightarrow  & (T^{\perp })^{ \vee } \otimes _{}\mathcal{O}_{\mathbb{P}^{n}}(d) &  \rightarrow  & \mathcal{N}_{f} &  \rightarrow  & 0 \\
\, & \, & \downarrow  & \, & \downarrow  & \, & \downarrow  & \, & \, \\
0 &  \rightarrow  & \mathcal{T}_{\mathbb{P}^{n}} &  \rightarrow  & \mathcal{T}_{f} &  \rightarrow  & \mathcal{N}_{f} &  \rightarrow  & 0 \\
\, & \, & \downarrow  & \, & \downarrow  & \, & \downarrow  & \, & \, \\
\, & \, & 0 & \, & 0 & \, & 0 & \, & \end{array}
\end{equation}
Let us consider the central vertical exact sequence, twisted by $\mathcal{O}_{\mathbb{P}^{n}}( -k)$ where $k$ is any integer:

\begin{equation}\label{eq:tgsheaves} 0 \rightarrow \mathcal{O}_{\mathbb{P}^{n}}( -k) \rightarrow (T^{\perp })^{ \vee } \otimes _{}\mathcal{O}_{\mathbb{P}^{n}}(d -k) \rightarrow \mathcal{T}_{f}( -k) \rightarrow 0. \end{equation}
The one has the following immediate result.
\begin{proposition}
The dimensions of the cohomology vector spaces $H^{i}\mathcal{T}_{f}( -k)$  are uniquely determined by the ranks and the degrees of the sheaves in the exact sequence (\ref{eq:tgsheaves}), with the only exceptions for $i \in \{n -1 ,n\}$ and $k \geq d +n +1.$
\end{proposition} 
In the cases not covered by the previous proposition we can consider the following exact sequence:

$$0 \rightarrow H^{n -1}\mathcal{T}_{f}( -k) \rightarrow H^{n}\mathcal{O}_{\mathbb{P}^{n}}( -k) \rightarrow H^{n}\mathcal{O}_{\mathbb{P}^{n}}(d -k) \otimes (T^{\perp })^{ \vee } \rightarrow H^{n}\mathcal{T}_{f}( -k) \rightarrow 0.$$

Of course, it suffices to calculate $h^{n -1}(\mathcal{T}_{f}( -k))$ toobtain also $h^{n}(\mathcal{T}_{f}( -k))$. We have the following

\begin{theorem}
\label{teotg} Let $f :\mathbb{P}^n \rightarrow \mathbb{P}^{s}$ be a map and let $\mathbb{P}(T)\subset\PP^n=\PP(U)$ be a suitable projective subspace such that, if $\nu_d\colon \PP(U)\to \mathbb{P}(S^dU)=\PP^N$ is the $d$-Veronese embedding and $\pi_T:\PP(S^dU)\dashrightarrow \PP^s=\PP(S^dU/T)$ is the projection with vertex $\PP(T)$, then $f=\pi_T\circ \nu_d$. 

Let us set $\dim (T) =e +1 =N -s$ and let us denote $\mathcal{T}_{f}=f^{ \ast }\mathcal{T}_{\mathbb{P}^{s}}$, the pull-back of the tangent bundle of $\mathbb{P}^{s}$. For any integer $k\geq d+n+1$, let us set $\chi  : =k -d -n -1 \geq 0.$ 
Then we have
$$h^{n -1}(\mathcal{T}_{f}( -k)) =\dim ( \partial ^{ -\chi }T) .$$
\end{theorem}

Before giving the proof of Theorem \ref{teotg}, we need

\begin{lemma}
$( \partial ^{ -t}T)^{ \bot } =T^{ \bot } \cdot S^{t}U^{ \vee }$ for any integer $t \geq 0.$\label{lemJPAA}

\begin{proof}
 For $t =0$ the equality is obvious. Let us assume that $t =1.$ Let us consider any two elements $\sigma  \in T^{\perp } \subset S^{d}U^{ \vee }$ and $\omega  =\alpha _{0} \partial _{0} + . . . +\alpha _{n} \partial _{n} \in U^{ \vee };$ we have that $\sigma \omega  \in ( \partial ^{ -1}T)^{ \bot }$, in fact, for any $q \in  \partial ^{ -1}T \subset S^{d +1}U$ we have: $ \langle q ,\sigma \omega  \rangle  =\sigma [\omega (q)] =0$ because $\omega (q) \in T$ , by definition of $ \partial ^{ -t}T$ and $\sigma  \in T^{\perp };$ therefore $T^{ \bot } \cdot U^{ \vee }$ $ \subseteq ( \partial ^{ -1}T)^{ \bot }$. Vice versa, let $q$ be any element of $[T^{ \bot } \cdot U^{ \vee }]^{\perp }$ $ \subseteq S^{d +1}U ,$ then, for any $\sigma  \in T^{\perp }$ and for any $i =0 , . . . ,n ,$ we have $ \langle q ,\sigma  \partial _{i} \rangle  =0 ,$ hence $\sigma [ \partial _{i}(q)] = \langle  \partial _{i}(q) ,\sigma  \rangle  =0 ,$ for any $i =0 , . . . ,n;$ this implies that $ \partial _{i}(q) \in T$ for any $i =0 , . . . ,n ,$ hence $q \in  \partial ^{ -1}T$ and therefore $[T^{ \bot } \cdot U^{ \vee }]^{\perp }$ $ \subseteq  \partial ^{ -1}T .$ This implies $T^{ \bot } \cdot U^{ \vee } \supseteq ( \partial ^{ -1}T)^{ \bot }$ and, as the other inclusion was proved before, we get $( \partial ^{ -1}T)^{ \bot } =T^{ \bot } \cdot U^{ \vee }$.

When $t \geq 2$ we can use recursion on $t$, starting from the case $t =1 ,$ since $T^{ \bot } \cdot S^{t +1}U^{ \vee } =(T^{ \bot } \cdot S^{t}U^{ \vee }) \cdot U^{ \vee }$ and $ \partial ^{ -t -1}T = \partial ^{ -1}( \partial ^{ -t}T) .$
\end{proof}

\end{lemma}

Now we can prove Theorem \ref{teotg}.

\begin{proof}
 (of Theorem \ref{teotg}) Note that the map

\begin{center}$H^{n}\mathcal{O}_{\mathbb{P}^{n}}( -k) \rightarrow H^{n}\mathcal{O}_{\mathbb{P}^{n}}(d -k) \otimes (T^{\perp })^{ \vee }$\end{center}\par
\noindent in the above sequence is the dual of the linear map

$$\varphi  :T^{\perp } \otimes H^{0}\mathcal{O}_{\mathbb{P}^{n}}(k -d -n -1) \rightarrow H^{0}\mathcal{O}_{\mathbb{P}^{n}}(k -n -1),$$
which is the same as $\varphi  :T^{\perp } \otimes S^{k -d -n -1}U^{ \vee } \rightarrow S^{k -n -1}U^{ \vee }$, which is a multiplication map, so that $\ensuremath{\operatorname*{Im}}(\varphi ) =T^{\perp } \cdot S^{k -d -n -1}U^{ \vee }$ and $h^{n -1}(\mathcal{T}_{f}( -k)) =\dim (S^{k -n -1}U^{ \vee }) -\dim (\ensuremath{\operatorname*{Im}}\varphi ) .$

By Lemma \ref{lemJPAA} we have:

$\dim (S^{k -n -1}U^{ \vee }) -\dim (\ensuremath{\operatorname*{Im}}\varphi ) =\dim (S^{k -n -1}U^{ \vee }) -\dim (( \partial ^{ -\chi }T)^{ \bot }) =\dim ( \partial ^{ -\chi }T)^{}$.
\end{proof}

\begin{corollary}
In the previous notation: if $\dim ( \partial ^{ -\chi }T) =0 ,$ then $h^{n -1}(\mathcal{T}_{f}( -k)) =0.$
\end{corollary}

\section{Cohomology of the normal bundle 
}
\label{Section normal bundle}
Now we want to consider the cohomology of the sheaf $\N_f$ defined as in (\ref{eq:exactTN}).
Recall, by Remark 2 of the previous section, that when $f$ is an embedding and in particular $X$ is smooth,  this is also the cohomology of the normal bundle $\N_X$ of $X$.

Let us consider the middle horizontal exact sequence in the diagram \ref{eq:diagramnormal}, twisted by $\mathcal{O}_{\mathbb{P}^{n}}( -k) ,$ where $k$ is any integer:

\begin{equation}\label{eq:normalsheaves}
0 \rightarrow U \otimes \mathcal{O}_{\mathbb{P}^{n}}(1 -k) \rightarrow (T^{\perp })^{ \vee } \otimes _{}\mathcal{O}_{\mathbb{P}^{n}}(d -k) \rightarrow \mathcal{N}_{f}( -k) \rightarrow 0.
\end{equation}
Similarly as in the previous section, where we studied the cohomology of the restricted tangent bundle, we have the following immediate result.
\begin{proposition}\label{prop:cohomfacile} The dimensions of the cohomology vector spaces $H^{i}\mathcal{N}_{f}( -k)$ are uniquely determined by the ranks and degrees of the sheaves appearing in the sequence (\ref{eq:normalsheaves}), with the only exceptions for $i \in \{n -1 ,n\}$ and $k \geq d +n +1.$  
\end{proposition}

In the cases not covered by the previous proposition, we consider the following exact sequence:

$$0 \rightarrow H^{n -1}\mathcal{N}_{f}( -k) \rightarrow U \otimes H^{n}\mathcal{O}_{\mathbb{P}^{n}}(1 -k)) \rightarrow (T^{\perp })^{ \vee } \otimes H^{n}\mathcal{O}_{\mathbb{P}^{n}}(d -k) \rightarrow H^{n}\mathcal{N}_{f}( -k) \rightarrow 0.$$

Of course, it suffices to calculate $h^{n -1}(\mathcal{N}_{f}( -k))$ to obtain also $h^{n}(\mathcal{N}_{f}( -k))$. We have

\begin{theorem}
\label{teo normale}  Let $f :\mathbb{P}^n \rightarrow \mathbb{P}^{s}$ be a map and let $\mathbb{P}(T)\subset\PP^n=\PP(U)$ be a suitable projective subspace such that, if $\nu_d\colon \PP(U)\to \mathbb{P}(S^dU)=\PP^N$ is the $d$-Veronese embedding and $\pi_T:\PP(S^dU)\dashrightarrow \PP^s=\PP(S^dU/T)$ is the projection with vertex $\PP(T)$, then $f=\pi_T\circ \nu_d$.  

Let $\N_f$ be the normal sheaf to the map $f$, as defined as in (\ref{eq:exactTN}). For any integer $k \geq d-n-1$, let us set $\chi  : =k -d -n -1 \geq 0$ and let us consider the map

\begin{center}$\psi _{\chi  ,k -n -2} :U \otimes S^{k -n -2}U \rightarrow S^{\chi }U \otimes S^{d}U.$\end{center}\par

 If $\chi  \geq 1$, then

\begin{eqnarray*}
H^{n -1}(\mathcal{N}_{f}( -k)) &=&\operatorname{Im}(\psi _{\chi  ,k -n -2}) \cap (S^{\chi }U \otimes T)\\
&=&(S^{\chi}U \otimes T) \cap \left(\bigcap_{i ,j =0...n;r ,s =0...n}\ker (D_{i ,j} \circ D_{r ,s})\right).\end{eqnarray*}

If $\chi  =0$, then $h^{n -1}(\mathcal{N}_{f}( -k)) =dim(\mu ^{ -1}(T))$ where $\mu  :U \otimes S^{d -1}U \rightarrow S^{d}U$ is the multiplication map.
\end{theorem}

\begin{proof}
Firstly let us assume $\chi  \geq 1$ and let us choose a base for $T^{\perp }$ $ \subset S^{d}U^{ \vee }$ and for $(T^{\perp })^{ \vee } \simeq S^{d}U/T ,$ say $T^{\perp } = \langle g_{0} , . . . ,g_{s} \rangle $ and $(T^{\perp })^{ \vee } = \langle g_{0}^{ \vee } , . . . ,g_{s}^{ \vee } \rangle  .$ Then the sheaf map

\begin{center}$U \otimes \mathcal{O}_{\mathbb{P}^{n}}(1) \rightarrow (T^{\perp })^{ \vee } \otimes _{}\mathcal{O}_{\mathbb{P}^{n}}(d)$\end{center}\par
\noindent can be described as follows:

$x_{1} \otimes l_{1} +x_{2} \otimes l_{2} + . . . +x_{n} \otimes l_{n}$ $ \rightarrow $

$ \rightarrow {\displaystyle\sum _{i =0...s}}g_{i}^{ \vee } \otimes l_{1} \partial _{ \partial _{1}}(g_{i}) +{\displaystyle\sum _{i =0...s}}g_{i}^{ \vee } \otimes l_{2} \partial _{ \partial _{2}}(g_{i}) + . . . . +{\displaystyle\sum _{i =0...s}}g_{i}^{ \vee } \otimes l_{n} \partial _{ \partial _{n}}(g_{i}$)

\noindent where $l_{j}$ are generic local sections of $\mathcal{O}_{\mathbb{P}^{n}}(1);$ recall that $g_{i} \in S^{d}U^{ \vee }$ hence it is a degree $d$ polynomial in $ \partial _{0} , . . . , \partial _{n} ,$ so that every $ \partial _{ \partial _{j}}(g_{i})$  is a degree $d -1$ polynomial in $ \partial _{0} , . . . , \partial _{n} .$ By Serre's duality we can identify 

$H^{n}(\mathbb{P}^{n} ,_{}\mathcal{O}_{\mathbb{P}^{n}}(1 -k)) \simeq H^{0}(\mathbb{P}^{n} ,_{}\mathcal{O}_{\mathbb{P}^{n}}(k -n -2))^{ \vee } \simeq S^{k -n -2}U$ and 

$H^{n}(\mathbb{P}^{n} ,_{}\mathcal{O}_{\mathbb{P}^{n}}(d -k)) \simeq H^{0}(\mathbb{P}^{n} ,_{}\mathcal{O}_{\mathbb{P}^{n}}(k -d -n -1))^{ \vee } \simeq S^{k -d -n -1}U$. 

In our case the sheaf map $\mathcal{O}_{\mathbb{P}^{n}}(1 -k) \rightarrow _{}\mathcal{O}_{\mathbb{P}^{n}}(d -k)$ is defined by global sections belonging to $H^{0}(\mathbb{P}^{n} ,_{}\mathcal{O}_{\mathbb{P}^{n}}(d -1)) \simeq S^{d -1}U^{ \vee } ,$ i.e. by the elements $ \partial _{ \partial _{j}}(g_{i})$ when we tensor by $(T^{\perp })^{ \vee }$, hence the induced map $H^{n}(\mathbb{P}^{n} ,_{}\mathcal{O}_{\mathbb{P}^{n}}(1 -k)) \rightarrow H^{n}(\mathbb{P}^{n} ,_{}\mathcal{O}_{\mathbb{P}^{n}}(d -k))$ can be considered as a linear map $S^{k -n -2}U \rightarrow S^{k -d -n -1}$$U$ acting as a differential operator of degree $d -1$ and the map $U \otimes H^{n}(\mathbb{P}^{n} ,_{}\mathcal{O}_{\mathbb{P}^{n}}(1 -k)) \rightarrow (T^{\perp })^{ \vee } \otimes H^{n}(\mathbb{P}^{n} ,_{}\mathcal{O}_{\mathbb{P}^{n}}(d -k))$ can be viewed as defined by the $(n +1)(s +1)$ differential operators $ \partial _{ \partial _{0}}(g_{i}) , . . . , \partial _{ \partial _{n}}(g_{i})$ for $i =0 , . . . ,s$.

Therefore we have that $H^{n -1}(\mathcal{N}_{f}( -k))$ is the kernel of the linear map

\begin{center}$\mu  :U \otimes S^{k -n -2}U \rightarrow (T^{\perp })^{ \vee } \otimes S^{k -d -n -1}U$\end{center}\par
\noindent defined by:

\begin{center}$x_{j} \otimes q \rightarrow {\displaystyle\sum _{i =0...s}}g_{i}^{ \vee } \otimes  \partial _{ \partial _{j}}(g_{i})(q)$ for any $q \in S^{k -n -2}U$, for any $j =0 , . . .n .$\end{center}\par
$\ker (\mu )$$ =\{$ $x_{0} \otimes q_{0} + . . . +x_{n} \otimes q_{n}$$ \in U \otimes S^{k -n -2}U$ \textbar{} $ \partial _{ \partial _{0}}(g_{i})(q_{0}) + . . . + \partial _{ \partial _{n}}(g_{i})(q_{n}) =0$ for any $i =0 , . . . ,s\} .$

This is equivalent to say that

$\{ \partial _{ \partial _{0}}(g)(q_{0}) + . . . + \partial _{ \partial _{n}}(g)(q_{n}) =0 ,$ as element of $S^{\chi }U ,$  for any $g \in T^{\perp }\} \Longleftrightarrow $\_

$\{p[ \partial _{ \partial _{0}}(g)(q_{0}) + . . . + \partial _{ \partial _{n}}(g)(q_{n})] =0$ for any $g \in T^{\perp }$, for any $p \in S^{\chi }U^{ \vee }\} \Longleftrightarrow $

$\{ \partial _{ \partial _{0}}(g)(p(q_{0})) + . . . + \partial _{ \partial _{n}}(g)(p(q_{n})) =0$ for any $g \in T^{\perp }$, for any $p \in S^{\chi }U^{ \vee }\} .$

Recall that $ \partial _{ \partial _{j}}(g)$ and $p$ act as derivations, hence they are commutative. 

Let us remark that $ \partial _{ \partial _{j}}(g)(h) =g(x_{j}h)$ for any $h \in S^{d -1}U$ if $g \in S^{d -1}U^{ \vee }$ as in this case; (note that $p(q_{i}) \in S^{d -1}U$), hence the above set coincides with

$\{g[x_{0}p(q_{0}) + . . . +x_{n}p(q_{n})] =0$ for any $g \in T^{\perp }$, for any $p \in S^{k -d -n -1}U^{ \vee }$ $\} \Longleftrightarrow $

$\{x_{0}p(q_{0}) + . . . +x_{n}p(q_{n}) \in T$ for any $p \in S^{k -d -n -1}U^{ \vee }$$\}$.

In conclusion: 

$H^{n -1}(\mathcal{N}_{f}( -k)) =\{x_{0} \otimes q_{0} + . . . +x_{n} \otimes q_{n} \in U \otimes S^{k -n -2}U\vert x_{0}p(q_{0}) + . . . +x_{n}p(q_{n}) \in T$ for any $p \in S^{\chi }U^{ \vee }\} =$

$ =\{x_{0} \otimes q_{0} + . . . +x_{n} \otimes q_{n} \in U \otimes S^{k -n -2}U\vert x_{0}p(q_{0}) + . . . +x_{n}p(q_{n}) \in T$ for any monomial $p \in S^{\chi }U^{ \vee }\} .$

Then $H^{n -1}(\mathcal{N}_{f}( -k)) =$

$ =\{x_{0} \otimes q_{0} + . . . +x_{n} \otimes q_{n} \in U \otimes S^{k -n -2}U\vert \psi _{\chi  ,k -n -2}(x_{0} \otimes q_{0} + . . . +x_{n} \otimes q_{n}) \in S^{\chi }U \otimes T\} =$

$ =\ensuremath{\operatorname*{Im}}(\psi _{\chi  ,k -n -2}) \cap (S^{\chi }U \otimes T)$.

In fact, if $x_{0}p(q_{0}) + . . . +x_{n}p(q_{n}) \in T$ for any $p \in S^{k -d -n -1}U^{ \vee } ,$ obviously $\psi _{\chi  ,k -n -2}(x_{0} \otimes q_{0} + . . . +x_{n} \otimes q_{n}) \in S^{\chi }U \otimes T .$ For the other direction recall that if $T = \langle \tau _{0} , . . . ,\tau _{e} \rangle $ then $\{\frac{d !}{(d +k) !}\binom{k}{I}x^{I} \otimes \tau _{0} , . . . ,\frac{d !}{(d +k) !}\binom{k}{I}x^{I} \otimes \tau _{e}\}$ with $\vert I\vert  =k ,$ is a base for $S^{\chi }U \otimes T ,$ hence if $\psi _{\chi  ,k -n -2}(x_{0} \otimes q_{0} + . . . +x_{n} \otimes q_{n}) \in S^{\chi }U \otimes T$ then $x_{0}p(q_{0}) + . . . +x_{n}p(q_{n}) \in T ,$ for any monomial $p \in S^{k -d -n -1}U^{ \vee }$, because $x_{0}p(q_{0}) + . . . +x_{n}p(q_{n})$ is a linear combination of $\tau _{0} , . . . ,\tau _{e} .$

\noindent By Proposition \ref{propIm-ker} we know that

\begin{center}$\ensuremath{\operatorname*{Im}}(\psi _{\chi  ,k -n -2}) ={\displaystyle\bigcap _{i ,j =0...n;r ,s =0...n}}\ker (D_{i ,j} \circ D_{r ,s}) .$\end{center}\par
Then $H^{n -1}(\mathcal{N}_{f}( -k)) =\ensuremath{\operatorname*{Im}}(\psi _{\chi  ,k -n -2}) \cap (S^{\chi }U \otimes T) =$

$ =(S^{\chi }U \otimes T) \cap \left({\displaystyle\bigcap _{i ,j =0...n;r ,s =0...n}}\ker (D_{i ,j} \circ D_{r ,s})\right) .$

Obviously if $\chi  =0$ (hence $k =d +n +1)$ the above formula cannot hold because $\psi \chi  ,k -n -2$ cannot be defined (recall 6) at {\textsection}2). However, by following the above proof, it is easy to see that $h^{n -1}(\mathcal{N}_{f}( -k)) =dim(\mu ^{ -1}(T))$ where $\mu  :U \otimes S^{d -1}U \rightarrow S^{d}U$ is the multiplication map.
\end{proof}

\begin{corollary}
If ${\displaystyle\bigcap _{i ,j =0...n;r ,s =0...n}}\ker (D_{i ,j} \circ D_{r ,s}) =0$ then $H^{n -1}(\mathcal{N}_{f}( -k)) =0.$
\end{corollary}

\begin{remark}
With the same technique one can also prove that 
$$H^{n -1}(\mathcal{T}_{f}( -k)) =(S^{\chi }U \otimes T) \cap \left({\displaystyle\bigcap _{i ,j =0...n}}\ker (D_{i ,j})\right) .$$
\end{remark}

Finally, in the case of an embedding $f:\PP^n\to\PP^s$, taking into account the discussion in Remark 2 of section 3, we get the following result.
\begin{corollary} 
\label{corip} Let $f:\PP^n\to \PP^s$ be an embedding, and let $X=\operatorname{Im}(f)$. Then the normal bundle $\N_X$ is identified with $\N_f$ and then Proposition \ref{prop:cohomfacile} and Theorem \ref{teo normale} provide the dimensions of the cohomology spaces $H^i\N_X(-k)$, for any $i,k$.
\end{corollary}

\section{The monomial case for $n =2.$
}
\label{Section n=2}

In this section we will assume that $f$ is embedding, hence in particular $X$ smooth, as in Corollary \ref{corip}. Moreover we will assume that $n =2 ,$ i.e. $X$ is a rational surface, and we will also assume that the map $f$ is {\em defined by monomials}.  In particular the vector space $T$ will be assumed to be generated by monomials. 

We remark that the same hypotheses in the case of curves, i.e. $\dim(X)=1$, have led to interesting applications (see \cite{A-R2} and \cite{A-R-T}).

When $n =2$ it is convenient to use different names for the three variables: $U = \langle x ,y ,u \rangle $ and $U^{ \vee } = \langle  \partial _{x} , \partial _{y} , \partial _{u} \rangle ;$ in this way $T$ is generated by $e +1$ distinct monomials $x^{\alpha }y^{\beta }u^{\gamma } \leftrightarrow (\alpha  ,\beta  ,\gamma )$ such that $\alpha  +\beta  +\gamma  =d$ and $\alpha  ,\beta  ,\gamma  \in [0 ,d]$.  As any one of these monomials is in fact identified by a triple of integers having the above properties, we can say that any vector space $T$ under consideration can be identified with a choice of $e +1$ triples $(\alpha  ,\beta  ,\gamma )$ such that $\alpha  +\beta  +\gamma  =d $ and $\alpha  ,\beta  ,\gamma  \in [0 ,d]$. Note that $f$ and $X$ are completely defined by choosing a suitable set $T$ of monomials as above.

To get a simple way to calculate the cohomology by hand, in most cases, it is useful to consider a graph $T\Delta$ whose vertices are the triples of integers such that $\alpha  +\beta  +\gamma  =d$ and $\alpha  ,\beta  ,\gamma  \in [0 ,d]$ and whose edges join every triple $(\alpha  ,\beta  ,\gamma )$ with $(\alpha +1 ,\beta  ,\gamma )$, $(\alpha  ,\beta+1  ,\gamma )$, $(\alpha  ,\beta  ,\gamma+1 )$ whenever this is possible taking into account that $\alpha  ,\beta  ,\gamma  \in [0 ,d]$.
This graph naturally assumes a triangular shape; we will call $\Delta$ this triangle: the three corners of $\Delta$ are $(d ,0 ,0) ,$ $(0 ,d ,0) ,$ $(0 ,0 ,d)$ and the three sides of $\Delta$ are the sets:
$\{(0 ,\beta  ,\gamma )| \beta+\gamma=d\}$, 
$\{(\alpha  ,0 ,\gamma )| \alpha+\gamma=d\}$, 
$\{(\alpha  ,\beta  ,0)| \alpha+\beta=d\}$.

In the next proposition we will show that, as $f$ has no base points, the set of triples of $\Delta$ defining $T$ is "far" from the corners of the triangle.

\begin{proposition}
\label{propP^2}Let $T$ $ \subset S^{d}U$ be an $(e +1)$-dimensional vector space defining a smooth projection of $\nu _{d}(\mathbb{P}^{2})$ in $\mathbb{P}^{s}$ as explained at the beginning of section \ref{Section tangent bundle}. Assume that the corresponding map $f$ is an embedding. Then, if we identify $T$ with a set of triples ${(\alpha  ,\beta  ,\gamma )}$ in $\Delta $  as above, we have: $\alpha  \leq d -2 ,$ $\beta  \leq d -2 ,$ $\gamma  \leq d -2$.
\end{proposition}

\begin{proof}
 The property is symmetric with respect to $\alpha  ,\beta  ,\gamma  ,$ so we can prove it only for $\alpha  .$ We have to show that $\alpha \leq d-2$ for any triple $(\alpha,\beta,\gamma) \in T$. First, let us show that $\alpha \leq d-1$.

Let us choose $(x^{d},x^{d -1}y , . . . ,u^{d})$ as a basis for $S^{d}U$ and let us choose coordinates in $\mathbb{P}(S^{d}U)$ with respect to this basis.  By contradiction, let us assume that $(d ,0 ,0) \in T ,$ hence $x^{d} \in T \subset S^{d}U ,$ and hence $(1 :0 : . . . :0) \in \mathbb{P}(T) \subset \mathbb{P}^{N}$$ =\mathbb{P}(S^{d}U)$ (here $N =\binom{d +2}{d}-1$). But this is not possible because $\mathbb{P}(T) \cap \nu _{d}(\mathbb{P}^{2}) = \varnothing $, as $X$ is smooth, while $(1 :0 : . . . :0) \in \nu _{d}(\mathbb{P}^{2})$ because $(1 :0 : . . . :0) = \nu _{d}(1 :0 :0)$ (recall that $(1 :0 :0)$ cannot be a base point by assumption).

Now let us show that $\alpha \leq d-2$. Let us assume, always by contradiction, that $(d -1 ,1 ,0) \in T .$ We know that $f(x :y :u) =( . . . :\lambda (\alpha  ,\beta  ,\gamma )x^{\alpha }y^{\beta }u^{\gamma } : . . .)$ is defined by all triples $(\alpha  ,\beta  ,\gamma )$ in $\Delta \backslash T ,$ with $\lambda (\alpha  ,\beta  ,\gamma )$ suitable non zero coefficients. By our assumption $(\alpha  ,\beta  ,\gamma ) \neq (d -1 ,1 ,0) .$ Let us consider the restriction $\phi$ of $f$ to the affine plane for which $x \neq 0$. In this affine plane we can choose coodinates $(v,w)$ by putting $v : =y/x$ and $w : =u/x$. In this  setting $\phi$ is an affine map such that $\phi(v ,w) =( . . . ,\lambda (\alpha = d-\beta-\gamma  ,\beta  ,\gamma )v^{\beta }w^{\gamma } , . . .)$. We have that:

if $\beta  \geq 2 ,$ then $\gamma  \geq 0$

if $\beta  =1 ,$ then $\gamma  \geq 1$

if $\beta  =0 ,$ then $\gamma  \geq 1$

\noindent because $(\beta  ,\gamma ) \neq (1 ,0) .$ Of course $(\beta  ,\gamma ) \neq (0 ,0)$ by the first part of the proof. 

Let us compute $\frac{ \partial \phi}{ \partial v} .$ Taking into account the above three possibilities, we get  that $\frac{ \partial \phi}{ \partial v}(0 ,0) =\underline{0}$ in any case. But this is not possible, because the smoothness of $X$ at $f(1 :0 :0) = \phi(0,0)$ implies that  $\frac{ \partial \phi}{ \partial v}(0 ,0)$ and $\frac{ \partial \phi}{ \partial w}(0 ,0)$ must be independent vectors, generating the tangent (affine) plane to $X$ at $f(1 :0 :0) = \phi(0,0) .$

The same argument works if we assume $(d -1 ,0 ,1) \in T$, so the Proposition is proved.
\end{proof}

\begin{remark}
Note that there exist examples of smooth surfaces $X$ for which

\begin{center}$T =\{$$(\alpha  ,\beta  ,\gamma )\vert \alpha  = d -2 , $$\beta  = d -2 , $$\gamma  = d -2\}$ \end{center}\par
Hence Proposition \ref{propP^2} is sharp. For instance let $d =3$ and

\begin{center}$f(x :y :u) =(x^{3} :x^{2}y :x^{2}u :xy^{2} :xu^{2} :y^{3} :y^{2}u :yu^{2} :u^{3})$\end{center}\par
\noindent with $T = \langle xyu \rangle $ and $\dim (T) =1.$ In this case $X$ is a smooth surface in $\mathbb{P}^{8}$ of degree $9$, as you can verify by using a computer algebra system as Macaulay.

\end{remark}

In the last part of this section we will state a Proposition showing that, for monomially embedded projective planes, in most cases, the calculation of $\partial T$ and $\partial^{-1} T$ become easier, making easier to compute the cohomology of the restricted tangent bundle.

Recall that the monomials generating $T$ in $\Delta$ are vertices ${\tau_i}$ of the graph $T\Delta$, hence we can define the distance between two monomials generating $T$ as the usual distance $\delta $ on a graph. We have the following

\begin{proposition}
\label{propgrafo} Let $T$ $ \subset S^{d}U$ be an $(e +1)$-dimensional vector space defining a smooth projection of $\nu _{d}(\mathbb{P}^{2})$ in $\mathbb{P}^{s}$ as explained at the beginning of section \ref{Section tangent bundle}. Assume that the corresponding map $f$ is an embedding. Let $T\Delta$ and $\Delta$ be the graphs defined as above. Let us denote $T = \langle \tau _{0} , . . .\tau _{e} \rangle $ and let us assume that $\delta (\tau _{i} ,\tau _{j}) \geq 2$ for any $i ,j$. Then:

i) $\dim ( \partial T) =3a +2b$, where $a$ is the number of generators of $T$ which are not on the sides of $\Delta$ and $b$ is the number of the other generators of $T$ ;

ii) $\dim( \partial ^{ -1}T)=0$.

\end{proposition}

\begin{proof}
i) Obviously, for a single monomial $(\alpha  ,\beta  ,\gamma ) \in T \subset \Delta  ,$ with $(\alpha  ,\beta  ,\gamma )$ not belonging to the three sides of $\Delta$, we have that $\dim ( \partial (\alpha  ,\beta  ,\gamma )) =3$; in fact $ \partial (\alpha  ,\beta  ,\gamma ) = \langle (\alpha  -1$, $\beta  ,\gamma ) ,(\alpha  ,\beta  -1 ,\gamma ) ,(\alpha  ,\beta  ,\gamma  -1) \rangle $.  Proposition \ref{propP^2} shows that $(\alpha,\beta,\gamma)$ cannot be a corner of $\Delta$. For a single monomial $(\alpha  ,\beta  ,\gamma ) \in $ $T$, belonging to any one of the three sides of $\Delta $, one has  $\dim ( \partial (\alpha  ,\beta  ,\gamma )) =2$.

In our assumptions every monomial of $T$ behaves as it were single; in other words: the contribution of action of $\partial$ over any monomial is not affected by the action over the other ones. Hence we have $\dim ( \partial T) =3a +2b$.

ii) For any single monomial $(\alpha  ,\beta  ,\gamma ) \in T \subset \Delta  ,$ we have that $ \partial ^{ -1}(\alpha  ,\beta  ,\gamma ) =0.$ To get examples for which $ \partial ^{ -1}T \neq 0 ,$ we need that, among the generators of $T$, there exists at least:

-  (first case) three monomials: {$(\alpha,\beta,\gamma), (\alpha - 1,\beta + 1,\gamma),(\alpha, \beta + 1,\gamma - 1)$} which are three vertices of $T\Delta$, not belonging to the sides of $\Delta$, whose mutual distance is 1 (giving rise to a subgraph of the triangular shape as the picture ``$\bigtriangledown$");

- (second case) a pair of monomials, whose distance is 1, lying on a side of $\Delta$.

In our assumptions every monomial of $T$ behaves as it were single (as in $i)$) and none of the above two cases can occur, so that $\dim( \partial ^{ -1}T)=0$.
\end{proof}

\begin{remark}
If in $T\Delta$ there exist pairs of generators for which $\delta  =1 ,$ then the calculation of $\dim ( \partial T)$ is more difficult.  In these cases, although the calculation can often be done by hand, a general formula is not available at the moment. As far as $\dim( \partial ^{ -1}T)$ is concerned,  if $T$ is the disjoint union of $h$ triples of the first case quoted in Proposition \ref{propgrafo} and $k$ pairs of the second case, then $\dim ( \partial ^{ -1}T) = h +k$. Otherwise the calculation is more difficult, as we said above. Note that if $T$ is a set of ``sparse'' monomials then $ \partial ^{ -1}T =0$.

\end{remark}

\section{An application to the normal sheaf of maps $f:\PP^2\to\PP^3$}\label{sec:nonsmooth}

In this section we will not assume that the map $f$ is an embedding. In this case $X$ can be singular, hence $\N_f$ cannot be the pull back of the normal bundle of $X$. However we have defined $\N_f$ as the cokernel of the sheaf map $df\colon T_{\PP^2}\to \T_f = f^\ast T_{\PP^3}$, so that there exists an exact sequence as the sequence (\ref{eq:exactTN}):
\begin{equation}\label{eq:exactN1}
0\to T_{\PP^2}\stackrel{df}\longrightarrow f^\ast T_{\PP^3}\to \N_f\to 0.
\end{equation}
Recall that one has the Euler sequences $0\to \OO_{\PP^2}\to\OO_{\PP^2}^3(1)\to T_{\PP^2}\to 0$ and $0\to \OO_{\PP^3}\to\OO_{\PP^2}^4(1)\to T_{\PP^3}\to 0$.
By pulling back the second one to $\PP^2$ by means of $f$, one obtains a simpler presentation of $\N_f$ given by
\begin{equation}\label{eq:exactN2} 
0\to \OO_{\PP^2}^3(1)\stackrel{F}\longrightarrow \OO^4_{\PP^2}(d)\to \N_f\to 0,
\end{equation} 
where $F$ is the map defined by the matrix $\left(\frac{\partial f_i}{\partial x_j}\right)$, (in the sequel we will denote by the same letters maps and matrices defining them). Note that the entries of $F$ are homogeneous polynomials of degree $d-1$ because $f$ is defined by polynomials of degree $d$.

Let us define $Z$ the subscheme of $\PP^2$ where $F$ drops rank, in the sequel $Z'$ will denote the $0$-dimensional part of $Z$, if any, when $dim(Z) = 1$. Note that $Z$ is also the branch locus of $f$.
\begin{example}\label{example:quadric}
Let $f\colon \PP^2\to \PP^3$ be given by $f(x:y:u)=(x^2:y^2:u^2:2xy)$ with $T=\langle xy,yu \rangle$ and $\dim(T)=2$.
Then 
$$F=2\left(\begin{array}{lll} x &0&0\\ 0&y&0\\0&0&u\\y&x&0\end{array}\right).$$
The closed set where $df$ drops rank is the set of common zeros of the $3\times 3$ minors of $F$, that is $Z=V(xyu,x^2u,y^2u)$. By the primary decomposition $(xyu,x^2u,y^2u)=(u)\cap (x,y)^2$, one sees that, as a scheme, $Z$ is the union of the line $l=V(u)$ with the fat point $P^{(2)}$ given by the first infinitesimal  neighborhood of the point $P=(0:0:1)$.  Denoting with $(t:s:v:w)$ the chosen coordinates of $\PP^3$, one sees that the image of $f$ is the cone $Y=V(4ts-w^2)$. 
Then one sees that the degree of $f\colon \PP^2\to Y$ is $\deg(f)=2$. The vertex of $Y$ is the point $Q=(0:0:1:0)$ and the schematic fiber of $f$ over $Q$ is $f^{\ast}(Q)=P^{(2)}$. This accounts for $P$ belonging to the branch locus of $f$. Moreover the pull-back of the conic $C=V(v,4ts-w^2)$ to $\PP^2$ is the line $l=V(u)$ counted with multiplicity $2$, and the restriction of $f$ to $l$ is the Veronese embedding of $l$ in $V(v)\cong\PP^2$, since it is defined by $f(x:y:0)=(x^2:y^2:0:2xy)$. In particular $f|_l\colon l\to C$ is injective, i.e. the cardinality of $f^{-1}(R)$ is $1$ for any point $R\in C$. This accounts for the inclusion of $l$ in the branch locus of $f$.
As the matrix of $F$ can be decomposed in two diagonal blocks, it is not difficult to see that, in this case, 
$\N_f = \OO_{l}(2) \oplus \mathcal{I}_{Z'}(4)$ 
where $Z'$ is the $0$-dimensional part of $Z$.

\end{example}
 
\subsection{Torsion of the normal sheaf}

We start recalling that the Eagon-Northcott complex associated to the map $\OO_{\PP^2}^3(1)\stackrel{F}\longrightarrow \OO^4_{\PP^2}(d)$ has the form
\begin{equation}\label{eq:eagnorth} 0\to \OO_{\PP^2}^3(1)\stackrel{F}\to \OO_{\PP^2}^4(d)\to \OO_{\PP^2}(4d-3)\to \OO_{Z}(4d-3)\to 0,
\end{equation} with the middle map in (\ref{eq:eagnorth}) defined by the $4$-tuple of $3\times 3$ minors of $F$. Here $Z$ is the subscheme defined  by the maximal minors.  Recall that (\ref{eq:eagnorth}) is in general only a complex, but that it is exact if the degeneracy locus has the expected dimension, in this case if $\dim Z=0$ (see \cite{G-P}).

Let $\N_f$ be the normal sheaf of the map $f:\PP^2\to \PP^3$ as in (\ref{eq:exactN1}) and (\ref{eq:exactN2}). Let us denote $\tau(\N_f)$ its torsion subsheaf and let $h_\N$ be the degree of the {\em divisorial part} of $\op{Supp}(\tau(\N_f))$, that is 
$$h_\N=c_1(\tau(N_f)).$$

We have the following result.
\begin{proposition}\label{prop:Ndecomp}
In the above situation there exists some $0$-dimensional subscheme $Z'\subset\PP^2$ such that $\N_f/\tau(\N_f)\cong \II_{Z'}(4d-3-h_\N)$.
\end{proposition}
\begin{proof} From the Eagon-Northcott complex (\ref{eq:eagnorth}) we have a map $\N_f\to \II_Z(4d-3)$, with $Z$ the degeneracy locus of $f$ endowed with the scheme structure defined by the maximal minors of $F$. As $\II_Z(4d-3)$ is torsion free, the map $\N_f\to \II_Z(4d-3)$ factors through $\N_f/\tau(\N_f)\to \II_Z(4d-3)$. Since the map $\OO_{\PP^2}^4(d)\to \II_{Z}(4d-3)$ induced by (\ref{eq:eagnorth}) is defined by the maximal minors of $F$ and it is surjective by the definition of $Z$, then also the map $\N_f/\tau(\N_f)\to \II_Z(4d-3)$ is surjective and, since $1=\rk(\N_f)=\rk(\N_f/\tau(\N_f))=\rk(\II_Z)$, the map $\N_f/\tau(\N_f)\to \II_Z(4d-3)$ is also injective, hence it is an isomorphism. Assume that the divisorial part of $Z$ has equation $H=0$, with $\deg H=h$, then one can write 
$$\N_f/\tau(\N_f)=\II_Z(4d-3)\cong \II_{Z'}(4d-3-h),$$
with $Z'$ a $0$-dimensional subscheme of $\PP^2$. To complete the proof of the proposition it is sufficient to show that $h=h_\N=c_1(\tau(\N_f))$. This is an immediate consequence of the exact sequences 
$$0\to\OO_{\PP^2}(1)^3\to\OO_{\PP^2}(d)^{4}\to \N_f\to 0$$
and

\begin{equation}\label{eq:NIZ'}
0\to \tau(\N_f)\to \N_f\to \II_{Z'}(4d-3-h)\to 0,
\end{equation}

\noindent which respectively imply $c_1(\N_f)=4d-3$ and $c_1(\N_f)=c_1(\tau(\N_f))+c_1(\II_{Z'}(4d-3-h))=h_\N+4d-3-h$, this last equality obtained by using that $\dim(Z')=0$. Hence $h_N=h$.  
\end{proof}

The following result relates the cohomology of $\N_f$ and that of $\II_{Z'}$. 
\begin{proposition}\label{prop:cohomcomp} Using notation as in section 4: for any $k\in \Z$ one has
\begin{eqnarray*}h^2\N_f(-k)&=&h^2\II_{Z'}(4d-3-h_\N-k)\\ &=&h^0\OO_{\PP^2}(k+h_\N-4d).\end{eqnarray*}
\end{proposition}
\begin{proof} By taking the long cohomology sequence of $$0\to \II_{Z'}(4d-3-h_\N-k)\to \OO_{\PP^2}(4d-3-h_\N-k)\to\OO_{Z'}(4d-3-h_\N-k)\to 0$$ and using $\dim Z'=0$ one obtains
$$h^2\II_{Z'}(4d-3-h_\N-k)=h^2\OO_{\PP^2}(4d-3-h_\N-k)=h^0\OO_{\PP^2}(k+h_N-4d),$$ the last equality being due to Serre duality. By the exact sequence
$$0\to \tau(\N_f)(-k)\to \N_f(-k)\to \II_{Z'}(4d-3-h_\N-k)\to 0,$$ due to Proposition \ref{prop:Ndecomp}, and using the fact that $\dim\op{Supp}(\tau(\N_f))\leq 1$, one obtains $$h^2\N_f(-k)=h^2\II_{Z'}(4d-3-h_\N-k).$$
\end{proof}

In particular, we can state the following result.
\begin{corollary} \label{cor:h}The function $h^2\N_f(-k)$ determines the degree $h_\N$ of the divisorial part of the branch locus of $f$.
\end{corollary}
Recall also that from the exact sequence (\ref{eq:exactN2}) one obtains the exact sequence
$$0\to H^1\N_f(-k)\to H^2\OO_{\PP^2}^3(1-k)\stackrel{F}\longrightarrow H^2\OO_{\PP^2}^4(d-k)\to H^2\N_f(-k)\to 0$$
and hence 
\begin{eqnarray*} h^2\N_f(-k)&=&4h^2\OO_{\PP^2}(d-k)-3h^2\OO_{\PP^2}(1-k)+h^1\N_f(-k)\\
&=& 4h^0\OO_{\PP^2}(k-d-3)-3h^0\OO_{\PP^2}(k-4)+h^1\N_f(-k).
\end{eqnarray*}

Then, from Theorem \ref{teo normale}, it follows immediately 
\begin{corollary}\label{cor formula h^2}
By using the same notation as in section 4: for any $k\geq d+4$ one has 
\begin{eqnarray*}h^2\N_f(-k)&=&4{k-d-1\choose 2}-3{k-2\choose 2}\\
&&+ \dim\left( (S^{\chi }U \otimes T) \cap [{\displaystyle\bigcap _{i,j=0,1,2;r,s=0,1,2}}\ker (D_{i ,j} \circ D_{r ,s})] \right),\end{eqnarray*}
where $\chi=k -d - 3$.
\end{corollary}

\begin{remark} It is well known that a finite morphism $f\colon X\to Y$ between smooth projective varieties of dimension $n$ has a pure $(n-1)$-dimensional branch locus, i.e. a divisor, whose first Chern class is $c_1(f^\ast K_Y^{-1}\otimes K_X)$. This is no longer true when $\dim X\not=\dim Y$, as in the cases  $f\colon\PP^2\to \PP^3$ taken in consideration. As we already saw in Example \ref{example:quadric}, the branch locus of $f$ is in general not pure, it may very well have components of different dimensions. In the next example we will show that it can even have embedded components. Nevertheless, it may be useful to can calculate some invariants of the branch locus of $f$. The above corollary says something on the divisorial part of that locus. When $\chi \leq 1$ the value of $h^2\N_f(-k)$ calculated by Theorem \ref{teo normale} depends only by the dimension of $T$ or by the cohomology of the first two bundles of (\ref{eq:exactTN}), hence it cannot depend on the branch locus of $f$, these values are in fact those predicted by the Eagon-Northcott complex associated to $F$. Hence the branch locus can be characterized by the values of $h^2\N_f(-k)$ only when $\chi \geq 2$.
\end{remark}

\begin{remark} If we consider a generalization of Example \ref{example:quadric} taking $f(x:y:u)=(x^2:y^2:u^2:p_2(x:y:u))$, where $p_2$ is a generic degree two homogeneous polynomial, then the branch locus $Z$ consists of $6$ points not in general position; these points belong to the cubic $xyu = 0 $. In this case $dim(Z) = 0$, the Eagon-Northcott complex is exact and the cohomology of $\N_f  = \mathcal{I}_Z(5)$ can be easily computed.

\end{remark}

\begin{example}\label{ex:moncubic} Consider $f\colon \PP^2\to \PP^3$ defined by $$f(x:y:u)=(x^3:y^3:u^3:3x^2y).$$
Then $d = 3$, $dim(T) = 7$, the homogeneous Jacobian is $$F=3\left(\begin{array}{lll} x^2&0&0\\ 0&y^2&0\\0&0&u^2\\ 2xy&x^2&0\end{array}\right)$$
and the row of the maximal minors defining the map $\OO_{\PP^2}^4(3)\to \OO_{\PP^2}(9)$ in (\ref{eq:eagnorth}) is 
$$(2xy^3u^2,x^4u^2,0,-x^2y^2u^2)=xu^2(2y^3,x^3,0,-xy^2).$$ 

Note that, if we slightly modify the example taking $f(x:y:u)=(x^3:y^3:u^3:3xyu)$ (or $f(x:y:u)=(x^3:y^3:u^3:p_{3}(x:y:u))$ with $p_3$ generic degree three homogeneous polynomial), then $dim(Z) = 0$ and the Eagon-Northcott complex is exact; on the contrary this is not the case in this example. Note also that $F$ can be decomposed into two diagonal blocks, but now this fact does not help in dividing $Z$ as the union of disjoint components of different dimensions, as in Example \ref{example:quadric}, because there are embedded components. This remark could be used to simplify the following diagram, but not in a substantial way. In fact
in this example $9=4d-3$ and the divisorial part of the subscheme $Z$ defined by the maximal minors is $V(xu^2)$, consisting in the divisor $L+2M$ with $L=\op{div}(x)$ a reduced line and $2M=\op{div}(u^2)$ a double line. Moreover $Z'=V(y^3,x^3,xy^2)$, a multiple point supported on the reduced line $L$, is an embedded primary component of $Z$.
Proposition \ref{prop:Ndecomp} predicts that $c_1(\tau(\N))=h_\N=3$ and that there exists an exact sequence $0\to \tau(\N_f)\to \N_f\to \II_{Z'}(6)\to 0$.

Indeed we have a commutative diagram where the vertical complexes are exact:
\[\begin{CD} 0@>>> \OO_{\PP^2}^3(1)@>G>> \OO_{\PP^2}(2)\oplus\OO_{\PP^2}(1)\oplus\OO_{\PP^2}(3)@>>> \OO_L(2)\oplus\OO_{2M}(3)@>>>0\\
&&@V\op{id}VV @VHVV @V\alpha VV\\
0@>>> \OO_{\PP^2}^3(1)@>F>> \OO_{\PP^2}^4(3)@>>> \N_f@>>>0\\
&&@VVV@VVV@VVV\\
&&0 @>>>\op{coker}(H)@>\cong>> \op{coker}(\alpha)@>>> 0
\end{CD}
\]
with $$G={\left(\begin{array}{lll} x&0&0\\0&1&0\\0&0&u^2\end{array}\right)},\quad H={\left(\begin{array}{lll} x&0&0\\0&y^2&0\\0&0&1\\ 2y&x^2&0\end{array}\right)}.$$
Now the complex 
$$0 \to \OO_{\PP^2}(2)\oplus\OO_{\PP^2}(1)\oplus\OO_{\PP^2}(3) \stackrel{H}\longrightarrow \OO_{\PP^2}^4(3)\stackrel{(2y^3,x^3,0,xy^2)}\longrightarrow \OO_{\PP^2}(6)\to\OO_{Z'}(6)\to 0
$$
is the Eagon-Northcott complex associated to $H$ and it is exact, since $Z'$ is $0$-dimensional. Hence $\op{coker}(\alpha)\cong \op{coker}(H)=\II_{Z'}(6)$ and therefore the rightmost column in the diagram above can be identified with $$0\to \tau(\N_f)\to \N_f\to \N_f/\tau(\N_f)\to 0.$$ In particular we see $$\tau(\N_f)\cong \OO_L(2)\oplus\OO_{2M}(3),$$ which verifies $c_1(\tau(\N_f))=c_1(\OO_L(2)\oplus\OO_{2M}(3))=3$. Moreover $deg(Z') = 7$; this number can be calculated directly from the equations of $Z'$ by a computer algebra system as Macaulay or from the cohomolgy of $\II_{Z'}(6)$, computed by using (\ref{eq:NIZ'}) and the cohomology of $\N_f$ computed as in Theorem \ref{teo normale}.
\end{example}

\subsection{Numerical bounds}

In this last subsection we give some numerical bounds for the degree $h_\N$ of the divisorial part of $Z$ and for the degree of its 0-dimensional part $Z'$. First of all note that the integer $h_\N$ can be easily bounded:  

\begin{proposition} \label{prop: h_N} By using the above notation we have: $0\leq h_\N\leq 3(d-1)$.
\end{proposition}

\begin{proof} $Z$ is defined by the $3\times 3$ minors of $F$, hence $h_\N$ cannot exceed the degree of such a minor.
\end{proof}

In Example \ref{example:quadric} one has $h_\N=1$, while $3d-3=3$.  In the following Proposition \ref{prop:d=2} we will show that indeed for $d=2$ the value $h_\N=1$ is the maximum possible value, by showing that, in that case, Example \ref{example:quadric} is the only possibility for $f:\PP^2\to\PP^3$ defined by quadratic polynomials and $h_\N>0$, up to the natural action of $PGL(3)\times PGL(4)$. In Example \ref{ex:moncubic} we have $h_\N=3$ and in that case $3d-3=6$. We currently do not know whether $h_\N=3$ is the maximum possible value when $f:\PP^2\to \PP^3$ is defined by cubic polynomials. This facts leads to the following

\begin{problem} Find an explicit better bound for $h_\N$.
\end{problem}

\begin{proposition} \label{prop:d=2} Let $f:\PP^2\to \PP^3$ be defined by quadratic polynomials and such that $f(\PP^2)$ is a quadric surface $S$, then the branch locus of $f$ has dimension $1$, and $f$ is $PGL(3)\times PGL(4)$-equivalent to the morphism described in Example \ref{example:quadric}. In particular the maximum possible value for $h_\N$ is $1$ for $d=2$.
\end{proposition}

\begin{proof}
Let us consider an arbitrary morphism $f\colon\PP^2\to \PP^3$ defined by homogeneous forms $f_0,\ldots,f_3$ of degree $2$. Writing as usual $f=\pi_L\circ\nu_2$, where $L$ is a suitable line in $\PP^5$. In order to study the branch locus of $f$, we need to consider how $L$ can intersect the secant variety of $\nu_2(\PP^2)$. We recall that $W:=Sec(\nu_2(\PP^2))=Tan(\nu_2(\PP^2))$ and it is  a cubic hypersurface in $\PP^5$, with defining polynomial the determinant of the general symmetric $3\times 3$ matrix and with singular locus equal to $\nu_2(\PP^2)$. 

As $f$ is a morphism, $L$ cannot intersect $\nu_2(\PP^2)$.  As the degree of $f$ must be $2$, for any generic point $P$ of $S$ the plane $\langle L,P \rangle$ must intersect $\nu_2(\PP^2)$ at two distinct points, so that, for any $P$, there exists a secant line of $\nu_2(\PP^2)$ cutting $L$. These secant lines cannot cut $L$ at a finite number of points (at most three, of course) because there are only a simple infinity of secant lines to $\nu_2(\PP^2)$ passing through any point $w$ of $W$, $w$ not in $\nu_2(\PP^2)$. Proof of this claim: in a suitable coordinate system $w$ has equation $xy=0$, it belongs only to the secant lines of $\nu_2(\PP^2)$ given by pairs of double lines passing through $(0:0:1)$, these lines correspond to pencils of conics of the following type: $\lambda(ax+by)^2+\mu(cx+dy)^2$, the pencil contains $xy=0$ if and only if $(ad)^2-(bc)^2=0$ and this is only one equation in $\PP^1 \times \PP^1$: the claim is proved.

Therefore $L$ is contained in $W$, $L$ does not intersect $\nu_2(\PP^2)$ and it corresponds to a pencil of singular conics: it is well known that there exists only one possibility: the conics of the pencil are given by a fixed line $l$ and a variable line in a pencil of lines whose center is not on $l$. By choosing suitably the coordinate system in $\PP^2$ we have: $L=\langle xu,yu \rangle$, and by choosing suitably a coordinate system in the target $\PP^3$ we have that $f$ is exactly as in Example \ref{example:quadric}.

Final step: if $L$ is not contained in $W$ then $deg(S)=4$ and it is easy to see that, in this case, $h_\N=0$. If $L$ is contained in $W$, as $L$ cannot intersect $\nu_2(\PP^2)$, the above argument shows that $f$ is exactly as in Example \ref{example:quadric}. There are no other possibilities, hence, in our assunptions, $h_\N$ is at most $1$ in any case.
\end{proof}

\begin{remark}  From Proposition \ref{prop:cohomcomp} we see that 
$$h^2\N(-k)=\left\{\begin{array}{lll} {k+h_\N-4d+2\choose 2} &{\rm for}& k+ h_\N\geq 4d\\
\\
h^2\N(-k)=0 &{\rm for}& k+h_\N<4d.\end{array}\right.$$
It follows that $4d-h_\N-1=\max(k\ |\ h^2\N_f(-k)=0)$. Recall also from Theorem  \ref{teo normale} and its proof that $h^2\N_f(-t)=0$ is equivalent to say that the map $$H^2\OO_{\PP^2}^3(1-t)\to H^2\OO_{\PP^2}^4(d-t)$$ is surjective,
equivalently that the map
$$\mu  :U \otimes S^{k -n -2}U \rightarrow (T^{\perp })^{ \vee } \otimes S^{k -d -3}U$$
is surjective.
This argument, or similar ones based on the formulas above, open the possibility to use the results of the previous sections to study the Problem above.
  
\end{remark}

As for the degree of $Z'$, for any map $f$ we can argue as in Example \ref{ex:moncubic} and write a commutative diagram as follows:

\[\begin{CD} 
0@>>> \OO_{\PP^2}^3(1)@>G>> \mathcal{E} @>>> \tau(\N_f)@>>>0\\
&&@V\op{id}VV @VHVV @V\alpha VV\\
0@>>> \OO_{\PP^2}^3(1)@>F>> \OO_{\PP^2}^4(d)@>>> \N_f@>>>0\\
&&@VVV@VVV@VVV\\
&&0 @>>> \mathcal{I}_{Z'}(4d-3-h_\N)@>\cong>> \mathcal{I}_{Z'}(4d-3-h_\N) @>>> 0
\end{CD}
\]

\noindent where $ \mathcal{E}$ is a suitable sheaf over $\PP^2$ and the rows and columns in the diagram are exact. 
By looking at the Chern classes of the sheaves involved by the diagram (here the Chern classes are identified with integers) and assuming to know $h_\N = c_1(\tau(\N_f))$ by the previous arguments we have:

$c_1( \mathcal{E}) = 3+h_\N$

$c_2(\OO_{\PP^2}^4(d)) = 6d^2 = c_2( \mathcal{E}) + c_2( \mathcal{I}_Z'(4d-3-h_\N)) + c_1( \mathcal{E})c_1(I_{Z'}(4d-3-h_\N)) = $

$ =  c_2( \mathcal{E}) + deg(Z') + (3 + h_\N)(4d-3-h_N).$

Unfortunately, we cannot say that $c_2( \mathcal{E}) \geq 0$ as in Example \ref{ex:moncubic}, so
that we have not a bound for $\deg(Z')$ from the above equations. However, see the following
remark, we can estimate $\deg(Z')$ by the cohomology of $\N_f$ as we have done for $h_\N$ in
Corollary \ref{cor:h}.

\begin{remark}  \label{remZ'} If we define $q:=4d-3-h_\N =d+\eta$ we have the following exact sequence from the above diagram:
$$0\to \tau(\N_f)(-q)\to \N_f(-q)\to \mathcal{I}_Z'\to 0$$

\noindent and we have:
$$deg(Z')=h^1(\mathcal{I}_{Z'})+1=1+h^1(\N_f(-q))-h^1(\tau(\N_f)(-q)) \leq 1+h^1(\N_f(-q)).$$

Note that our method is useful only when $q\geq d+5$ ($\eta \geq 5$), otherwise the calculation of $h^1(\N_f(-q))$ is immediate from \ref{eq:exactN2}.
  
\end{remark}

\subsection*{Acknowledgements}We thank the anonymous referee for her/his detailed and helpful remarks on the first submitted version of this article.

\end{document}